\documentclass[12pt,a4paper,oneside]{amsart}

\usepackage{a4wide}

\usepackage{graphicx} 
\usepackage[a4paper,margin=1.15in]{geometry}
\usepackage{amsmath,amssymb,amsthm,mathtools}
\usepackage{bbm}
\usepackage{enumitem}
\usepackage{xfrac}

\usepackage{xcolor}

\usepackage[nopatch=footnote]{microtype}
\usepackage{hyperref}
\usepackage{booktabs}
\usepackage{lmodern}

\theoremstyle{plain}
\newtheorem{theorem}{Theorem}[section]
\newtheorem{proposition}[theorem]{Proposition}
\newtheorem{lemma}[theorem]{Lemma}
\newtheorem{corollary}[theorem]{Corollary}

\theoremstyle{definition}

\numberwithin{equation}{section}
\numberwithin{figure}{section}
\numberwithin{table}{section}

\theoremstyle{remark}
\newtheorem{remark}[theorem]{Remark}
\newcommand{\R}{\mathbb{R}}
\newcommand{\E}{\mathbb{E}}
\newcommand{\dd}{\textup{d}}
\newcommand{\calM}{\mathcal{M}}
\newcommand{\calP}{\mathcal{P}}
\newcommand{\N}{\mathbb{N}}

\newcommand{\Fhat}[1]{\widehat{#1}}
\newcommand{\diam}{\textup{diam}}
\newcommand\res{\mathop{\hbox{\vrule height 7pt width .5pt depth 0pt \vrule height .5pt width 6pt depth 0pt}}\nolimits}

\newcommand{\expAhlfors}{\beta}

\usepackage{calc}
\usepackage{ragged2e}
\usepackage{setspace}
\newcommand{\FundingLogos}{%
  \raisebox{0pt}{\includegraphics[height=1.5cm]{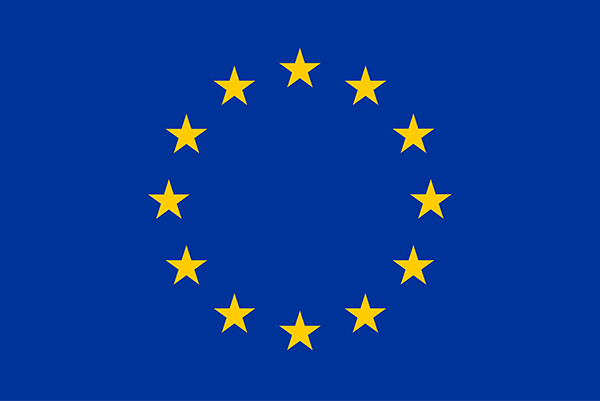}}%
  \hspace{1em}%
  \raisebox{0pt}{\includegraphics[height=1.5cm]{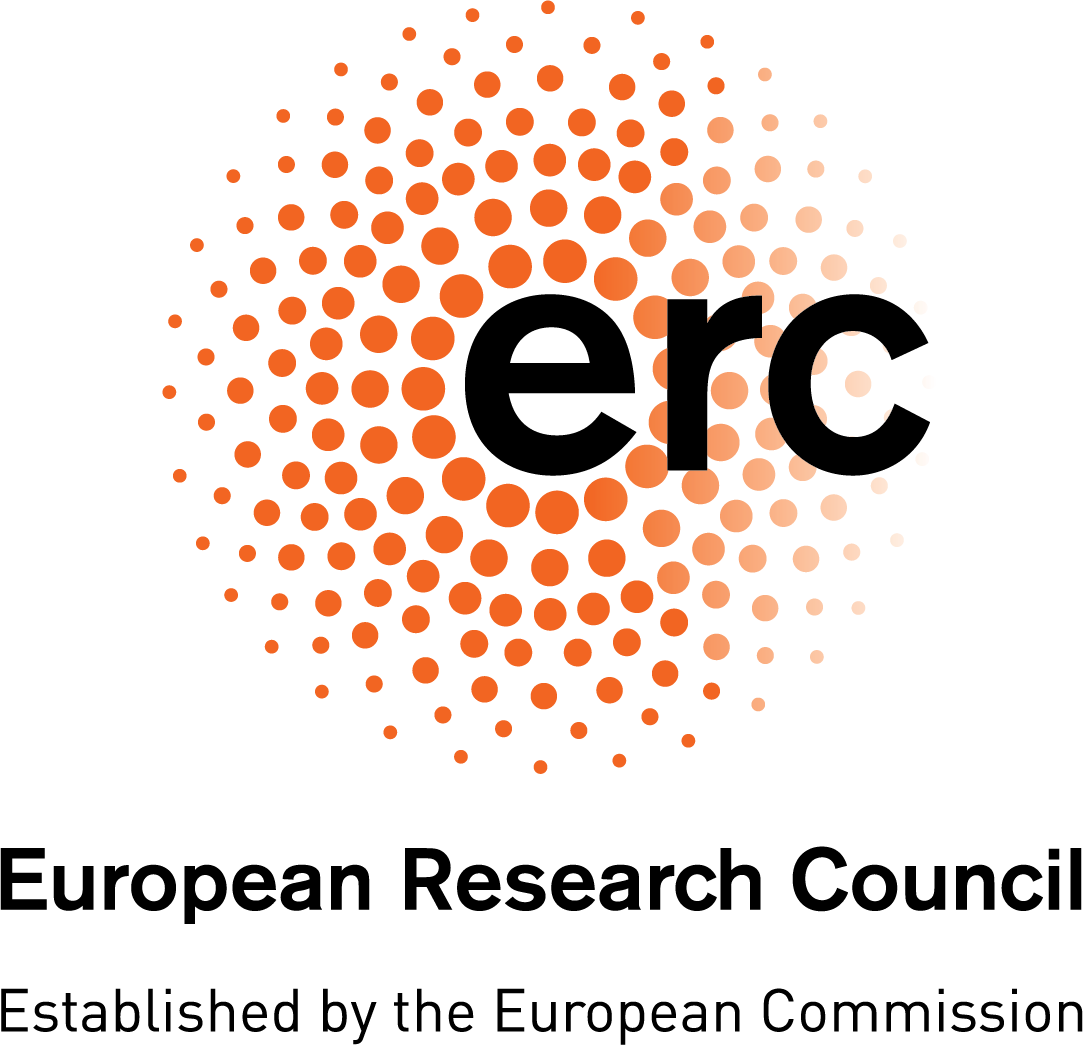}}%
}

\usepackage{tcolorbox}
\tcbset{colback=white,boxrule=0.5pt,arc=2pt,left=4pt,right=4pt,top=4pt,bottom=4pt}

\title{Sharp Rates of MMD Empirical Estimation with Power Kernels}
\author{Francesco Colasanto}
\address[Francesco Colasanto]{Department of Mathematics, CIT School, Technical University of Munich, Munich, Germany}
\email{francesco.colasanto@tum.de}
\author{Matteo Focardi}
\address[Matteo Focardi]{DiMaI U. Dini, Università di Firenze, Florence, Italy}
\email{matteo.focardi@unifi.it}
\author{Massimo Fornasier}
\address[Massimo Fornasier]{Department of Mathematics, CIT School, Technical University of Munich, Munich \& Munich Center for Machine Learning (MCML), Munich, Germany}
\email{massimo.fornasier@ma.tum.de}
\author{Francesco Mattesini}
\address[Francesco Mattesini]{Department of Mathematics, CIT School, Technical University of Munich, Munich \& Munich Center for Machine Learning (MCML), Munich, Germany}
\email{francesco.mattesini@tum.de}
\date{May 2026}

\thanks{F.\,C., M.\,F. and F.\,M. acknowledge the support of the ERC Advanced Grant NEITALG, grant agreement No. 101198055.}

\begin{document}

\begin{abstract}
We establish quantitative rates of convergence for the empirical estimation
of probability measures by means of 
the Maximum Mean Discrepancy (MMD) with
power kernel $K_q(x,y) = -|x-y|^q$, $q \in (0,2)$. 
The resulting
discrepancy is the classical \emph{energy distance}
$$\mathcal E_q^2(\mu, \omega) = -\frac{1}{2}\iint_{\mathbb{R}^d \times \mathbb{R}^d} |x-y|^q \, d(\mu - \omega)(x)\, d(\mu - \omega)(y),$$
and we ask how fast the best $N$-point empirical approximation
$\inf_{\mu_N \in \mathcal{P}^N}\mathcal{E}_q(\mu_N,\omega)$ decays as
$N \to \infty$. Given a probability measure $\omega$ on $\mathbb{R}^d$
with compact support satisfying an Ahlfors regularity condition of exponent $\beta \in (0,d]$, we prove
that the sharp two-sided bound
$$\mathcal E_q(\mu_N, \omega) \asymp N^{-\frac{1}{2}\left(1 + \frac{q}{\beta}\right)}$$
holds both for the worst-case empirical measure $\mu_N$ (lower bound,
holding for every configuration of $N$ points) and for an optimally chosen
empirical measure $\mu_N$ (upper bound). This complements the qualitative
consistency result of Fornasier and H\"utter \cite{fornasier2014consistency},
who proved narrow convergence of the minimizers of $\mathcal E_q^2(\cdot,
\omega)$ over empirical measures without quantitative rates. 
\end{abstract}

\maketitle

\section{Introduction}

\subsection{Measure quantization via MMD with power kernel and the main result}\label{sec:intro_MMD}

\emph{Measure quantization} is a central problem in approximation theory, numerical integration, and modern machine learning. Given a target probability measure $\omega$ on $\R^d$ and an integer $N\geq 1$, one seeks an empirical measure
\begin{equation*}
    \mu_N \;=\; \frac{1}{N}\sum_{i=1}^N \delta_{x_i}, \qquad x_1,\dots,x_N\in\R^d,
\end{equation*}
that approximates $\omega$ as accurately as possible with respect to a chosen discrepancy. Different choices of discrepancy lead to different quantization theories. The classical quantization theory, going back to Zador \cite{zador1966topics} adopts the $p$-Wasserstein distance, see also \cite{GrafLuschgy} for a throughoug review. In the last two decades the machine-learning, statistics, and image-processing communities \cite{gretton2006kernel, sejdinovic2013equivalence, schmaltz2010electrostatic} have increasingly turned to a different family of discrepancies, namely the \emph{Maximum Mean Discrepancy} (MMD). Given a symmetric kernel $K:\R^d\times\R^d\to\R$, the squared MMD between two probability measures $\mu,\omega\in\mathcal P(\R^d)$ is
\begin{equation}\label{eq:mmd-intro}
    \mathrm{MMD}_K^2(\mu,\omega)
    \;=\;
    \iint_{\R^d\times\R^d} K(x,y)\,d(\mu-\omega)(x)\,d(\mu-\omega)(y),
\end{equation}
which is non-negative whenever $K$ is positive (or conditionally negative) definite. A particularly important and analytically rich subfamily of MMD, and the one we shall focus on in this paper, is obtained by choosing the \emph{power kernel}
\begin{equation*}
    K_q(x,y) \;=\; -|x-y|^q,\qquad q\in(0,2),
\end{equation*}
which, by a classical result of Schoenberg \cite{schoenberg1938monotone,schoenberg1938positive}, is conditionally negative definite on $\R^d$, and therefore yields a non-negative discrepancy on probability measures. The resulting quantity, classically known as the \emph{energy distance}, is
\begin{equation}\label{eq1}
  \mathcal{E}_q^2(\mu,\omega)
  \;=\;
  -\frac{1}{2}
  \iint_{\R^d\times\R^d}
  |x-y|^q\,d(\mu-\omega)(x)\,d(\mu-\omega)(y),
  \qquad 0 < q < 2,
\end{equation}
defined as soon as $\mu,\omega$ have finite $q$-moments, and coincides up to a factor of $1/2$ with $\mathrm{MMD}_{K_q}^2(\mu,\omega)$. The choice $q=1$ recovers the original electrostatic halftoning energy of Schmaltz et al.~\cite{schmaltz2010electrostatic} and lies at the heart of the energy statistics introduced by Sz\'ekely \cite{szekely2007measuring,szekely2013energy,szekely2023energy}; we refer to Section~\ref{sec:histreview} below for a detailed historical review. As we shall discuss in Section~\ref{sec:compENvsWas}, the use of $\mathcal{E}_q$ in lieu of the Wasserstein distance offers decisive computational and statistical advantages in high dimensions, while still inducing a comparable topology on $\mathcal{P}(\R^d)$, see Proposition~\ref{prop:edtop} for a precise statement.

The natural quantitative question for this quantization problem is then:
\medskip

\centerline{\emph{How fast does the quantization error $\inf_{\mu_N\in\mathcal P^N}\mathcal{E}_q(\mu_N,\omega)$ decay as $N\to\infty$?}}
\medskip

where $\mathcal{P}^N$ denotes the class of $N$-point empirical measures. Fornasier and H\"utter \cite{fornasier2014consistency} established that minimizers of $\mathcal{E}_q(\cdot,\omega)$ on $\mathcal P^N$ converge narrowly to $\omega$, but no quantitative rate of convergence was provided. The main contribution of this paper is to fill this gap by establishing \emph{sharp two-sided bounds} on this infimum, under the natural assumption that $\omega$ is Ahlfors regular of exponent $\beta\in(0,d]$. Our main result can be stated as follows.
\begin{theorem}\label{t:main-intro}
Assume $\omega\in \mathcal{P}(\R^d)$ such that $\mathrm{spt}\, \omega$ is connected, and that there exist $C_\omega > 1,R_0>0$ and $\expAhlfors\in(0,d]$ such that 
\begin{equation}\label{e:alfhors reg0}
C^{-1}_\omega r^\expAhlfors \leq \omega(B_r(x)) \leq C_\omega r^\expAhlfors 
\end{equation}
for every $x\in \mathrm{spt}\, \omega$ and every $r\in (0,R_0)$. Fix $q\in (0,2)$ and the constant $\kappa_{d,q} = 2\pi^{d/2}\Gamma(1-q/2)/(2^q\Gamma((d+q)/2))>0$. Then the following two statements hold:
\begin{itemize}
  \item[(i)]   there exist constants $c_L>0$ and $N_0\in\N$ such that for every $N\ge N_0$ and every empirical measure $\mu_N=\frac1N\sum_{i=1}^N\delta_{x_i}$,
\begin{equation}\label{eq:Hdots-lower0}
\mathcal{E}_q(\mu_N,\omega) = \frac{1}{\sqrt{\kappa_{d,q}}}\|\mu_N-\omega\|_{H_0^{-s}} \ \ge\ c_L\,N^{-\frac{1}{2}(1+\frac{q}{\expAhlfors})};
\end{equation}
\item[(ii)]   there exist constants $C_L > 0$ and $N_0 \in \mathbb{N}$ such that for every $N\geq N_0$ there exist $x_1,\dots,x_N \in \textup{spt} \; \omega$ such that 
\begin{equation}\label{e:upper-bound-main}
\mathcal{E}_q(\mu_N,\omega) = \frac{1}{\sqrt{\kappa_{d,q}}}\|\mu_N-\omega\|_{H_0^{-s}}\le\ C_L\,N^{-\frac{1}{2}(1+\frac{q}{\expAhlfors})},
\end{equation}
where ${\mu_N}:=\frac1N\sum_{i=1}^N \delta_{x_i}$.
\end{itemize}
\end{theorem}
As an immediate consequence we obtain the sharp asymptotics
\begin{equation*}
\min_{\mu \in \mathcal P^N} \mathcal{E}_q(\mu,\omega)  \asymp N^{-\frac{1}{2}(1+\frac{q}{\expAhlfors})}, \quad N \to \infty.
\end{equation*}
Before proceeding, we comment on two notable features encoded in the exponent $-\frac{1}{2}(1+q/\expAhlfors)$. First, the rate depends on the \emph{intrinsic dimension} $\expAhlfors$ of $\mathrm{spt}\,\omega$ and not on the ambient dimension $d$: when $\omega$ is supported on a low-dimensional set (e.g.\ a smooth submanifold or a self-similar fractal), the convergence is genuinely faster than the worst-case dimension would suggest. Second, the rate $N^{-\frac{1}{2}(1+q/\expAhlfors)}$ is strictly faster than the i.i.d. rate $N^{-1/2}$ obtained by sampling $N$ points from $\omega$ (see equation \eqref{eq:stat} below and the discussion therein): optimal quantization beats Monte Carlo, and quantitatively so.

The proof of Theorem~\ref{t:main-intro} rests on two main ingredients. A key one is a revisitation of Schoenberg's embedding theory (see Theorem~\ref{t:equivalenza} below): we establish that
\begin{equation}\label{eq:equiv}
  \mathcal E_q^2(\mu, \omega)
  \;=\; \kappa_{d,q}^{-1}
  \|\mu - \omega\|^2_{{H}^{-s}_0}
  \;=\;
\kappa_{d,q}^{-1} \int_{\R^d}
  |\hat{\mu}(\xi) - \hat{\omega}(\xi)|^2
  \,|\xi|^{-(d+q)}\, d\xi,
  \qquad s = \tfrac{d+q}{2},
\end{equation}
where ${H}^{-s}_0(\R^d)$ is the homogeneous Sobolev spaces of negative exponent $-s$ \cite{auricchio2026kinetic, fornasier2014consistency, szekely2023energy},
identifying the squared energy distance with the squared homogeneous Sobolev norm of negative exponent of the signed measure $\mu-\omega$. Armed with this identification, the lower bound \eqref{eq:Hdots-lower0} is proved by a duality argument in the pre-Hilbert space $(\mathcal{M}_0(\R^d),\|\cdot\|_{\dot H^{-s}_0})$, using carefully constructed bump functions adapted to the geometry of $\omega$. The upper bound \eqref{e:upper-bound-main} is proved probabilistically by constructing an equimass partition of $\mathrm{spt}\,\omega$ and sampling one point per cell, exploiting the Ahlfors regularity to control the diameter of the cells. As a by-product, we obtain the sharp \emph{stochastic} asymptotics
\begin{equation*}
\E [\mathcal{E}^2_q(\mu_N,\omega)]  \asymp N^{-(1+\frac{q}{\expAhlfors})}, \quad N \to \infty,
\end{equation*}
for the random empirical measure obtained by sampling one point per cell of an equimass partition (see formula \eqref{e:random empirical} and Theorem~\ref{t:upper bound in media}), thus connecting the statistical convergence rate of $\mathcal E_q$ with classical numerical integration / quantization theory \cite{GrafLuschgy}. The connectedness assumption on the support of $\omega$ is used in the upper bound inequality only and can be ralaxed if $q \in (0, \expAhlfors)$, see Corollary~\ref{c:upper senza connesso}.\\

We now proceed to a detailed historical account of the energy distance, situating our contribution within several converging streams of research.

\subsection{Energy distances and their historical development}\label{sec:histreview}
The energy distance \eqref{eq1} has roots in multiple distinct mathematical traditions of the twentieth century, which gradually converged over the decades.
 
The mathematical starting point is the classical \emph{Riesz energy theory} of the
1920s--30s. The $\alpha$-Riesz capacity of a compact set is built around the notion of the
$\alpha$-energy of a measure $\mu$, defined as
$E_\alpha(\mu) = \iint |x-y|^{-\alpha}\,d\mu(x)\,d\mu(y)$ for $\alpha > 0$.
This is the classical \emph{repulsive} (singular) case, opposite in sign to the
kernel $|x-y|^q$ with $q > 0$.
In modern language, the Riesz energy with external field $V$ is
\[
  \mathcal{E}^V_s(\mu)
  \;=\;
  \iint \!\bigl(K_s(x-y)+V(x)+V(y)\bigr)\,d\mu(x)\,d\mu(y),
\]
where, for $-2 < s < d$, the kernel is $K_s = s^{-1}|\cdot|^{-s}$
(with the convention $K_0 = -\log|\cdot|$).
The special case $s = d-2$ recovers the Newton--Coulomb kernel.
When $s < 0$, i.e.\ when the exponent $q = -s > 0$, the kernel becomes
\emph{attractive} rather than repulsive, and one arrives precisely at the
kernel $|x-y|^q$ appearing in $\mathcal{E}_q$.
 
The decisive abstract step was taken by Schoenberg in two landmark
papers published in 1938 \cite{schoenberg1938monotone,schoenberg1938positive}.
Schoenberg's theorem asserts that $e^{-t\rho(\cdot,\cdot)}$ is a positive-definite kernel for
every $t > 0$ if and only if $\rho(\cdot,\cdot)$ is a \emph{conditionally
negative-definite} kernel.
He also established that a metric space embeds isometrically into a Hilbert space
if and only if its distance function is conditionally negative definite.
The key example, proved in those same papers, is that $|x-y|^q$ for $0 < q < 2$
is conditionally negative definite on $\R^d$, i.e.
\[
  \sum_{i,j} c_i c_j |x_i - x_j|^q \;\leq\; 0
  \qquad\text{whenever } \sum_i c_i = 0.
\]
This is precisely the algebraic condition that guarantees $\mathcal{E}_q(\mu,\omega)\geq 0$
whenever $\mu$ and $\omega$ are probability measures (or more generally signed measures
with equal total mass).
For $q \geq  2$ the kernel $|x-y|^q$ ceases to be conditionally negative definite,
and the functional $\mathcal{E}_q$ can take negative values; in the limiting case
$q = 2$ one finds a degenerate expression related to the variance, while for $q > 2$
no metric structure on the space of probability measures is obtained in general, see \cite[Chapter 3, Section 2, Corollary 2.10]{BergChristensenRessel1984}.
 
In the 1950s--60s, Jacques Deny (1950) and subsequently Bent Fuglede
(1960) developed the functional-analytic theory of energy spaces for Radon measures.
Deny introduced the pre-Hilbert space of signed measures with finite Riesz energy,
and Fuglede proved its completeness properties.
A central result in this line of work is that the pre-Hilbert space of Radon measures
on $\R^n$ with finite weak $\alpha$-Riesz energy embeds isometrically into its
completion, which is a Hilbert space of tempered distributions with energy defined via
the Fourier transform.
This Fourier-analytic perspective reveals that $\mathcal{E}_q(\mu,\omega)$ is, up to
a positive constant, the squared norm of $\mu - \omega$ in a reproducing-kernel
Hilbert space, with reproducing kernel $k(x,y) = -|x-y|^q$
(which is positive definite on the subspace of measures with zero total mass
precisely because $|x-y|^q$ is conditionally negative definite).
 
The connection to mathematical statistics was forged independently by G\'abor J.\ Sz\'ekely in the mid-1980s.
During a series of lectures delivered in 1984--85 at the Hungarian Academy of Sciences
in Budapest and at MIT, Yale, and Columbia, Sz\'ekely introduced the class of
\emph{energy statistics} (E-statistics), motivated by an analogy with Newton's
gravitational potential energy: statistical observations are thought of as heavenly bodies
governed by a potential energy that equals zero if and only if a given statistical null
hypothesis is true.
The squared energy distance between two probability distributions $F$ and $G$
on $\R^d$, with independent random vectors $X, X' \sim F$ and $Y, Y' \sim G$, is
\[
  D^2(F,G)
  \;=\;
  2\,\mathbb{E}\|X-Y\|^q
  \,-\,\mathbb{E}\|X-X'\|^q
  \,-\,\mathbb{E}\|Y-Y'\|^q,
  \qquad q\in(0,2).
\]
Writing this in measure-theoretic terms immediately gives
$D^2(\mu,\omega) = -\iint |x-y|^q\,d(\mu-\omega)(x)\,d(\mu-\omega)(y)$,
which is exactly $2\,\mathcal{E}_q^2(\mu,\omega)$.
The energy distance satisfies all the axioms of a metric on the space of probability
measures: it is non-negative, symmetric, satisfies the triangle inequality, and equals
zero if and only if $F = G$.
The formal results were published and developed in subsequent work, including the
monograph \cite{szekely2023energy} and the influential paper \cite{szekely2013energy}  and  the paper \cite{szekely2007measuring}, which introduced \emph{distance covariance} and
\emph{distance correlation} as independence measures derived from the same framework.
 
A parallel and independently developed line of work arose in the machine-learning
community under the name of \emph{Maximum Mean Discrepancy} (MMD) \cite{gretton2006kernel,gretton2012kernel}.
In this setting one defines, for a characteristic kernel $k$,
\[
  \mathrm{MMD}^2(\mu,\omega)
  \;=\;
  \iint K(x,y)\,d(\mu-\omega)(x)\,d(\mu-\omega)(y).
\]
The choice $K(x,y) = -|x-y|^q$ (restricted to zero-mean signed measures) yields
$\mathrm{MMD}^2 = 2\,\mathcal{E}_q^2$, so the energy distance is a special case of the
MMD with a Riesz-type kernel.
The equivalence between distance-based and kernel-based methods for hypothesis testing
was clarified by several authors in the 2000s and 2010s \cite{sejdinovic2013equivalence,lyons2013distance,shen2021exact}, firmly unifying the statistical and machine-learning perspectives.
 
In summary, the functional $\mathcal{E}_q(\mu,\omega)$ sits at the intersection of
classical potential theory (Riesz, 1920s), the theory of conditionally negative-definite
kernels and Hilbert-space embeddings (Schoenberg, 1938), the functional analysis of
energy spaces (Deny--Fuglede, 1950--60s), and the modern theory of energy statistics
(Sz\'ekely, 1984--present) and kernel methods (MMD, 2000s).
The thread connecting all of these is the conditional negative definiteness of
$|x-y|^q$ for $0 < q < 2$, a property discovered in abstract form by Schoenberg and
subsequently re-interpreted from completely independent starting points by
statisticians and the machine-learning community alike.

After 2010, seemingly without full knowledge of this previous history, the energy distance \eqref{eq1} (for $q=1$) was independently re-discovered in the image processing community as a tool to provide image halftoning (or image quantization). 

Schmaltz et al. ~\cite{schmaltz2010electrostatic} introduced electrostatic halftoning
as a practical image processing method: given a grayscale image representing a probability
measure $\omega$, they place $N$ point masses by minimizing an energy that balances pairwise
repulsion between the points and attraction toward the image's intensity.
Their discrete functional is exactly
\[
  \mathcal{E}_N^2(x)
  \;=\;
  \frac1N\sum_{i=1}^N \int |x_i - y|\,d\omega(y)
  \;-\;
  \frac{1}{2N^2}\sum_{i,j} |x_i - x_j|,
\]
which is a particle discretization of the continuous energy above with $q=1$.
The paper is entirely computational and heuristic: it works well visually, but
no convergence or consistency theory is provided.
 
Teuber et al.~\cite{teuber2011dithering} then gave this functional a rigorous
variational foundation. They recast the halftoning problem as the minimization of a
difference-of-convex-functions (DC) functional, proved analytical properties such as
coercivity, characterized one-dimensional minimizers explicitly, and proposed a
forward-backward splitting algorithm accelerated by the non-equispaced fast Fourier
transform (NFFT). The Fourier acceleration is significant: evaluating $|x_i - x_j|^q$
for $N$ particles naively costs $O(N^2)$, but the NFFT reduces this to $O(N \log N)$.
However, Teuber et al.\ still do not address whether the discrete minimizers actually
converge to $\omega$ as $N \to \infty$, i.e.\ whether the method is \emph{consistent}.
\\
This issue was the starting point of the work by Fornasier and H\"utter~\cite{fornasier2014consistency}. Their main contribution is a proof that the empirical measures
$\mu_N^* = \frac1N\sum_{i=1}^N \delta_{x_i}$, supported on the minimizers of
$\mathcal{E}_N^2$, converge narrowly to $\omega$ as $N \to \infty$.
The key mathematical tool is $\Gamma$-convergence of the discrete functionals to a
continuous target functional. The central difficulty is that the continuous energy is
not lower semi-continuous with respect to the narrow topology and is not even
well-defined on all probability measures. They resolve both issues via a suitable
Fourier representation: for $\psi = |\cdot|^q$ with $1 \le q < 2$, the energy
equals (up to a constant)
\[
  \widehat{\mathcal{E}}[\mu]
  \;=\;
  \kappa_{d,q}^{-1}
  \int_{\R^d}
  \bigl|\hat{\mu}(\xi) - \hat{\omega}(\xi)\bigr|^2
  \,|\xi|^{-d-q}\,d\xi,
\]
which is manifestly non-negative, lower semi-continuous, and zero if and only if
$\mu = \omega$.
 
This Fourier formula is precisely the energy distance (or MMD with a
Riesz-type kernel $K(x,y) = -|x-y|^q$) in its spectral form, as discussed earlier in the
energy distance literature (see also ~\cite[Chapter 3, Section 2, Corollary 2.10]{BergChristensenRessel1984}): the squared energy distance between $\mu$ and $\omega$
can always be written as a weighted $L^2$ distance between their Fourier transforms,
with weight $|\xi|^{-d-q}$~\cite{szekely2013energy}. The connection to the theory
of conditionally negative-definite kernels~\cite{schoenberg1938monotone,schoenberg1938positive}
is explicit in the paper, which relies on the fact that $|\cdot|^q$ for $0 < q < 2$
is a conditionally negative-definite kernel, the property first identified by
Schoenberg~\cite{schoenberg1938monotone,schoenberg1938positive} and later exploited
systematically by Sz\'ekely and Rizzo~\cite{szekely2007measuring,szekely2013energy}
in the statistical theory of energy distances. 

The related paper~\cite{chauffert2017projection} addresses the following general problem:
given a probability measure $\omega$ on $\Omega \subset \R^d$, find its best approximation
in a set $\mathcal{M}_N$ of Radon measures, where the discrepancy is measured by
\[
  \inf_{\mu \in \mathcal{M}_N}
  \bigl\| h * (\mu - \omega) \bigr\|_2^2,
\]
with $h \in L^2(\Omega)$ a fixed convolution kernel.
 
The choice of the kernel $h$ is crucial.
When $h$ is taken to be the kernel associated with the energy distance $|\cdot|^{q/2}$,
the functional reduces exactly to the energy distance between $\mu$ and $\omega$,
directly connecting this work to the framework of
Fornasier--H\"utter~\cite{fornasier2014consistency} and to the broader literature
on energy statistics~\cite{szekely2013energy,sejdinovic2013equivalence}.
The central theoretical contribution is a necessary and sufficient condition on the
sequence $(\mathcal{M}_N)_{N \in \N}$ that guarantees weak convergence of the
projections $\mu^*_N \to \omega$ as $N \to \infty$.
This condition is of approximation type: the sequence of sets must be rich enough
to approximate, in the weak sense, any probability measure.
The proof relies on tools from functional analysis and measure theory, in particular
on the weak-$*$ compactness of the space of Radon measures.
 
The authors show that this abstract projection problem encompasses, as special cases,
two concrete image rendering tasks.
The first is stippling, i.e.\ representing an image with $N$ point masses
(Dirac deltas), where $\mathcal{M}_N$ is the set of empirical measures supported on
$N$ atoms: in this case the problem reduces exactly to the electrostatic halftoning
of Schmaltz et al.~\cite{schmaltz2010electrostatic} and to the measure quantization
framework of Fornasier--H\"utter~\cite{fornasier2014consistency}.
The second is continuous line drawing, where the image is represented by a
single continuous curve of bounded length: here $\mathcal{M}_N$ is the set of measures
that are absolutely continuous with respect to arc-length measure on curves of length
$\leq N$, a much richer and geometrically structured family.
 
On the computational side, the authors propose a forward-backward splitting algorithm
to solve the discretized version of the problem.
As in Teuber et al.~\cite{teuber2011dithering}, the non-equispaced fast Fourier
transform (NFFT) is used to reduce the computational cost from $O(N^2)$ to
$O(N \log N)$. Numerical results include the automatic synthesis of artistic drawings
(single-stroke and stippled reproductions of images), demonstrating the versatility
of the framework.
 
In the broader picture, this work can be seen as a unifying generalization of the
preceding contributions: whereas Schmaltz et al.\ and Teuber et al.\ focus on sets
of discrete measures (stippling), and Fornasier--H\"utter prove consistency for that
specific case, Chauffert et al.\ extend the theory to arbitrary structured sets of
measures and provide abstract convergence conditions. The connection with the energy
distance is explicit: by Parseval's theorem, the functional $\|h* (\mu - \omega)\|_2^2$
is a weighted $L^2$ distance between the Fourier transforms of $\mu$ and $\omega$,
\[
  \bigl\| h* (\mu - \omega) \bigr\|_2^2
  \;=\;
  (2\omega)^{-d}
  \int_{\R^d}
  |\hat{h}(\xi)|^2\,
  \bigl|\hat{\mu}(\xi) - \hat{\omega}(\xi)\bigr|^2
  \, d\xi,
\]
which is precisely the structure of the spectral representation
\[
  \widehat{\mathcal{E}}[\mu]
  \;=\;
  \kappa_{d,q}^{-1}
  \int_{\R^d}
  |\xi|^{-d-q}\,
  \bigl|\hat{\mu}(\xi) - \hat{\omega}(\xi)\bigr|^2
  \, d\xi
\]
derived by Fornasier--H\"utter~\cite{fornasier2014consistency}
(corresponding to the choice $|\hat{h}(\xi)|^2 \propto |\xi|^{-d-q}$),
and identified as the energy distance / MMD with a Riesz kernel by
Sejdinovic et al.~\cite{sejdinovic2013equivalence} and
Sz\'ekely--Rizzo~\cite{szekely2013energy}.
\\

To complement this strand of work, the papers \cite{fornasier2013consistency, difrancesco2014asymptotic} addressed the Wasserstein gradient flow of MMD and energy distances for $\omega$ fixed. In particular \cite{fornasier2013consistency} did consider the gradient flow of MMD with sufficiently smooth kernels, while \cite{difrancesco2014asymptotic} did address the one of energy distances on measures on the one dimensional space. 
On the more applied side, Hertrich, Wald, Altekr\"uger and Hagemann~\cite{hertrich2024generative}
use Wasserstein gradient flows of the Riesz-kernel MMD with the same exponent range
$q\in(0,2)$ to build generative models, approximating a fixed target $\omega$ by the
evolution of $N$ interacting particles. The accuracy of such schemes is governed precisely
by how well an optimized $N$-point empirical measure approximates $\omega$ in $\mathcal{E}_q$,
so the sharp quantization rate of Theorem~\ref{t:main-intro} would provide a quantitative
foundation for the particle counts these methods require. Notably, their analysis relies on
the dimension-free i.i.d. MMD sample complexity rate $N^{-1/2}$ in~\eqref{eq:stat}, which our result shows can be
strictly improved by optimal quantization.
Building on the one-dimensional analysis, Duong, Stein, Beinert, Hertrich and
Steidl~\cite{duong2026wasserstein} complemented the results of \cite{difrancesco2014asymptotic} by characterizing the flow of the negative-distance-kernel MMD
($q=1$) towards a fixed target $\nu$ through an associated Cauchy problem on quantile functions,
from which they derive explicit solution formulas together with invariance and smoothing
properties.
These results have been recently greatly generalized on the $d$-dimensional torus $\mathbb T^d$  by the paper~\cite{chizat2026quantitative}, which studied the well-posedness and long-time
behaviour of the Wasserstein gradient flow of the Kernel Mean Discrepancy (KMD,
also known as MMD) functional. The model case considered is the squared homogeneous Sobolev distance
\[
  \mathcal{E}_s^2(\mu,\omega)
  \;=\;
  \frac{1}{2}\,\|\mu - \omega\|_{{H}_0^{-s}(\mathbb T^d)}^2,
  \qquad s \geq 1,
\]
 with $\omega$ a fixed target probability measure.
This functional corresponds to a pairwise Riesz kernel interaction with exponent
$q = 2s - d$, and its Wasserstein gradient flow is the continuity equation
\[
  \partial_t \mu_t
  \;+\;
  \nabla \cdot \bigl(\mu_t \, v_t\bigr) = 0,
  \qquad
  v_t = -\nabla \frac{\delta \mathcal{E}_s^2(\cdot,\omega)}{\delta \mu}[\mu_t],
\]
where the velocity field $v_t$ is determined by the first variation of
$\mathcal{E}_s^2(\cdot,\omega)$.
 
The paper establishes three main results.
First, inspired by Yudovich's theory for the two-dimensional Euler equation,
the authors prove global existence and uniqueness of the flow in natural weak
regularity classes, without assuming absolute continuity of the initial measure.
Second, in the critical case $s = 1$ (corresponding to the Coulomb or Newtonian
kernel), they show that the flow converges \emph{globally} to $\omega$ at an
\emph{exponential} rate under minimal assumptions on the data.
Third, for $s > 1$, they establish \emph{local} convergence at explicit
\emph{polynomial} rates of the form $O(t^{-\alpha})$, where the exponent $\alpha$
depends explicitly on $s$ and on the Sobolev regularity of both $\mu_0$ and $\nu$.
These rates hold both at the energy level and in higher Sobolev regularity classes,
and are shown to be sharp when $\nu$ is the uniform measure.
Crucially, except for the special case $s = 1$, even non-quantitative convergence
had been an open question in all these settings prior to this work.
 
As a notable application, the authors consider the gradient flow of the population
loss for shallow neural networks with ReLU activation in the infinite-width and
continuous-time limit. This dynamic can be recast as a Wasserstein--Fisher--Rao
gradient flow on the space of nonneg\-ative measures on the sphere $\mathbb{S}^d$,
and exploiting a correspondence with the Sobolev energy case with $s = (d+3)/2$,
the authors derive an explicit polynomial local convergence rate for this training
dynamics.
 
The paper also includes numerical experiments in dimension $d = 1$ using both PDE
(finite-volume) and particle discretizations, which illustrate and validate the
theoretical convergence analysis.
 
In the broader context of the energy distance literature, this work can be seen as
the most quantitative and general convergence result available for the Wasserstein
gradient flow of the energy
\[
  \mathcal{E}_q^2(\mu, \omega)
  \;=\;
  -\frac{1}{2}
  \iint_{\mathbb T^d \times \mathbb T^d}
  |x - y|^q \, d(\mu - \omega)(x)\, d(\mu - \omega)(y),
\]
whose Fourier representation $\mathcal{E}_q^2(\mu,\omega)$ is precisely
$\|\mu - \omega\|_{H_0^{-s}(\mathbb T^d)}^2$ for $q = 2s - d$. The gradient flow of this
functional governs both the long-time asymptotics of interacting particle systems
used in image halftoning~\cite{schmaltz2010electrostatic,teuber2011dithering} and
the training dynamics of overparameterized neural networks, thus unifying several
streams of research through a single analytical framework.\\

We conclude this historical account with the recent paper \cite{auricchio2026kinetic}, which is a comparative review to trace three families of divergences, namely Wasserstein distances, Fourier-based metrics, and energy distances, back to their origins in the kinetic theory of rarefied gases (Boltzmann equation, convergence to Maxwellian equilibrium). It argues that the same analytical tools developed to measure how fast a gas relaxes to equilibrium naturally give rise to the divergences most used in machine learning, and it establishes their mutual equivalence in the high-dimensional setting. 

\subsection{Comparison of energy distances and Wasserstein distances in terms of computational cost}\label{sec:compENvsWas}
 Given two empirical measures $\mu_N = \frac1N\sum_{i=1}^N \delta_{x_i}$ and
$\omega_M = \frac{1}{M}\sum_{j=1}^M \delta_{y_j}$, computing the exact $p$-Wasserstein
distance $W_p(\mu_N, \omega_M)$ requires solving a linear program with $N \times M$
variables and $N + M$ equality constraints. The exact cost is $O(N^3 \log N)$ using
the Hungarian or auction algorithm for the balanced case $N = M$, and more generally
$O(N^2 M + N M^2)$ for the unbalanced case. This is prohibitively expensive for
large $N$.\\
When it comes to statistical convergence, the dimension $d$ enters in a subtle but
devastating way. For a fixed measure $\omega$ with enough moments bounded, the expected Wasserstein error between the empirical
measure $\mu_N$ built from $N$ i.i.d.\ samples from $\omega$ is well-known to be
\[
  \E[W_p^p(\mu_N, \omega)] \;\asymp\; N^{-p/d},
  \qquad d > 2 p,
\]
see 
Dereich-Scheutzow-Schottstedt \cite{dereich} and
Fournier-Guillin \cite{fournier2015rate} for a precise statement. This means that to achieve a fixed approximation error $\varepsilon$ one
needs $N \asymp \varepsilon^{-d}$ samples, an exponential dependence on $d$.
In one dimension the Wasserstein distance admits the closed-form representation given by the Dall'Aglio formula \cite{dallaglio}
\[
W_p^p(\mu, \nu)
\;=\;
  \int_0^1 \bigl|F_\mu^{-1}(t) - F_\nu^{-1}(t)\bigr|^p \, dt,
\]
where $F_\mu^{-1}$ and $F_\nu^{-1}$ are the quantile functions, and the empirical
version can be computed in $O(N \log N)$ by sorting. In higher dimensions no such
closed form exists.
Several approximations reduce the cost.
The \emph{Sinkhorn algorithm} \cite{cuturi2013sinkhorn,peyre2019computational} regularizes the transport problem with
an entropic penalty $\varepsilon H(\mu)$, reducing the cost to $O(N^2)$ per
iteration with typically $O(N^2 / \varepsilon^2)$ overall complexity, but the approximation error and numerical stability degrade as
$\varepsilon \to 0$. The \emph{sliced Wasserstein distance} approximates $W_2$ by
averaging one-dimensional Wasserstein distances over $L$ random projections, at
cost $O(LN\log N)$, but introduces a bias that grows with $d$. Multiscale and
hierarchical methods can achieve near-linear complexity in favorable geometric
situations, but their guarantees remain dimension-dependent.
For general probability measures $\mu$ and $\omega$ on $\R^d$, computing $W_p(\mu,\omega)$
analytically is possible only in special cases: one-dimensional distributions (via
quantile functions), Gaussian measures in any dimension (via the Bures--Wasserstein
metric). In general, the computation requires solving
an infinite-dimensional linear program (Kantorovich's formulation), which has no
closed-form solution and must be approximated numerically by discretizing the measures
and solving the resulting finite linear program, thus reducing to the empirical case
with all its associated costs.
\\ 

Now, given two empirical measures $\mu_N = \frac1N\sum_{i=1}^N \delta_{x_i}$ and
$\omega_M = \frac{1}{M}\sum_{j=1}^M \delta_{y_j}$, the squared energy distance is
\[
  \mathcal E_q^2(\mu_N, \omega_M)
  \;=\;
  \frac{1}{NM}\sum_{i,j} |x_i - y_j|^q
  \;-\;
  \frac{1}{2 N^2}\sum_{i,j} |x_i - x_j|^q
  \;-\;
  \frac{1}{2 M^2}\sum_{i,j} |y_i - y_j|^q.
\]
The naive cost of evaluating each double sum is $O(N^2)$ (for $N = M$), making the
total cost $O(dN^2)$, already substantially cheaper than exact Wasserstein computation.
Moreover, since the kernel $|x-y|^q$ is radially symmetric, the
\emph{non-equispaced fast Fourier transform} (NFFT) can reduce the cost to
$O(dN \log N)$, as exploited by Teuber et al.~\cite{teuber2011dithering} and
Chizat et al.~\cite{chizat2026quantitative}.
\\ 
The statistical convergence
rate of the empirical energy distance is dramatically better than that of the
Wasserstein distance. In this paper we  show that for suitable sampling by constructing an equimass partition of $\mathrm{spt}\,\omega \Subset \R^d$ and sampling one point per cell
\[
  \E\bigl[\mathcal E_q^2(\mu_N, \omega)\bigr] \asymp  N^{-(1+\frac{q}{\expAhlfors})}
\]
holds  without strong dependence on the ambient dimension $d$. Moreover, if $\mu_N$ is generated from $N$ i.i.d. samples of $\omega$, a direct computation yields
\begin{equation}\label{eq:stat}
    \E\bigl[\mathcal E_q^2(\mu_N, \omega)\bigr] = \frac{1}{2 N} \E\bigl[ |X- X'|^q] \asymp N^{-1},
\end{equation}
as long as $\int_{\R^d} |x|^q d\omega(x) < \infty$, see also \cite{szekely2013energy}\footnote{The same computation gives $\mathbb{E}\bigl[\mathcal{E}_q^2(\mu_N,\nu_M)\bigr] = \frac{1}{2}\left(\frac1N+\frac{1}{M}\right)\mathbb{E}_{X,Y\overset{\text{i.i.d.}}{\sim}\omega}\bigl[|X-Y|^q\bigr]$ as soon both $\mu_N$ and $\nu_M$ are generated  with $N$ and $M$ i.i.d. samples of the same $\omega \in \mathcal P_q(\R^d)$ (matching problem). }. This \emph{dimension-free} statistical
convergence is one of the principal practical advantages of energy distances over
Wasserstein distances in high-dimensional settings, and it is a key reason why
MMD-based methods have become widespread in machine learning.
In particular, the $N^{-1/2}$ convergence rate of the empirical MMD, combined with the $O(N^2)$ cost 
of its U-statistic estimator (cf. \cite[Lemma~6 and Eq.~(5)]{gretton2012kernel}), 
provides the analytical foundation for the kernel two-sample tests of 
Gretton et al.~\cite{gretton2012kernel,gretton2006kernel}. This framework, and its subsequent refinements, have found extensive 
application in nonparametric hypothesis testing and distribution comparison, 
particularly in high-dimensional regimes where the curse of dimensionality 
renders the empirical Wasserstein distance statistically inefficient
\cite{cheng2024kernel, yan2023kernel}.
\\

We further recall that for general probability measures $\mu$ and $\omega$, the squared energy distance
admits the closed-form Fourier representation \eqref{eq:equiv}.
This means that for measures with known or computable Fourier
transforms (e.g.\ Gaussians, stable distributions, mixtures), the energy distance
can be evaluated analytically or semi-analytically without any optimization
whatsoever, in stark contrast to the Wasserstein distance. For Gaussian measures
$\mu = \mathcal{N}(m_1, \Sigma_1)$ and $\omega = \mathcal{N}(m_2, \Sigma_2)$, the
energy distance has a closed-form expression via the Fourier formula, up to eventually approximating the integral.
\\
 
The following table summarizes the comparison.
{ \footnotesize
\medskip
\begin{center}
\begin{tabular}{lcc}
\toprule
 & Wasserstein $W_p$ & Energy distance $\mathcal E_q$ \\
\midrule
Empirical, exact cost          & $O(N^3 \log N)$       & $O(dN^2)$               \\
Empirical, fast approximation  & $O(N^2)$ (Sinkhorn)   & $O(dN\log N)$ (NFFT)    \\
Statistical rate (empirical)   & $N^{-p/d}$ (curse of dim.) & $N^{-1/2}$ (dim.-free) \\
Non-empirical, general         & Infinite-dim.\ LP, no closed form
                                                        & Closed form via Fourier \\
Non-empirical, Gaussian        & Bures metric 
                                                        & Closed form via Fourier \\
Sensitivity to $d$             & Exponential (statistics) & Polynomial (computation) \\
\bottomrule
\end{tabular}
\end{center}
}
\medskip

 The relationship between Wasserstein distances and Fourier-based discrepancies has
also been investigated from a different viewpoint in the kinetic and imaging
literature. In particular, in
\cite{Auricchio2020} the authors study suitable Fourier-based
probability metrics, originally motivated by kinetic theory, and prove quantitative
equivalence estimates with $W_1$ and $W_2$ in a discrete setting, with explicit
constants. Their results provide further evidence that Fourier-based discrepancies
can be much more convenient computationally while retaining a precise relation with
transport distances. More broadly, the recent review
\cite{auricchio2026kinetic} discusses Wasserstein distances, Fourier-based metrics
and energy distances as part of a common family of divergences arising naturally
from kinetic theory.

The comparison proved in Section~\ref{sec:W1} is complementary to this line of
work. It is tailored specifically to the power-kernel MMD considered here and shows
that, for compactly supported probability measures,
$\mathcal E_q^2(\mu,\nu)\leq W_1^q(\mu,\nu)$, up to the normalization convention
fixed in Theorem~\ref{t:equivalenza}. We also characterize the equality cases
exactly: equality holds only in the trivial case $\mu=\nu$ or when both measures are
Dirac masses.

The fundamental conclusion is that energy distances are computationally and
statistically far more tractable than Wasserstein distances in high dimensions,
at the cost of being a slightly weaker metric. 
The topological relationships between $\mathcal{E}_q$ and Wasserstein distances 
are investigated in Modeste and Dombry~\cite[Theorem 13]{modeste2024}, who prove that, for 
every $q\in(0,2)$, convergence in $\mathcal{E}_q$ is strictly sandwiched between 
$W_{q/2}$-convergence (stronger) and $W_r$-convergence for all $r<q/2$ 
(weaker). In particular, $\mathcal{E}_q$ is a strictly weaker notion of 
convergence than $W_{q/2}$. 
In Section~\ref{sec:W1} we sharpen this picture by establishing 
the explicit quantitative bound 
$\mathcal{E}_q^2(\mu,\nu)\leq W_1^q(\mu,\nu)$ for compactly 
supported measures, with a complete characterization of the equality cases (cf. Theorem~\ref{thm:W1}). 
Since $W_{q/2}\leq W_1^{q/2}$ by Jensen's inequality, our estimate is 
controlled by a Wasserstein distance of order $1$ rather than $q/2$, providing 
a strictly stronger bound than the one implicit in \cite[Theorem~13]{modeste2024}.
In the converse direction, our lower 
bound in Theorem~\ref{t:main-intro}\,(i) shows that no dimension-free Hölder 
estimate of the form $W_1(\mu_N,\omega)\leq C\,\mathcal{E}_q(\mu_N,\omega)^\theta$ 
with $\theta$ independent of $\beta$ can hold uniformly over $N$-point empirical 
configurations, since $\mathcal{E}_q(\mu_N,\omega)$ and $W_1(\mu_N,\omega)$ 
decay at intrinsic rates that depend on the dimension $\expAhlfors$ of 
$\mathrm{spt}\,\omega$. This is consistent with the dimension-dependent 
exponents appearing in \cite[Propositions~16 and 17]{modeste2024}.
For completeness, we include the following 
self-contained characterization of convergence in $\mathcal{E}_q$ in terms of weak convergence and uniform 
$q$-moment boundedness.
\begin{proposition}\label{prop:edtop}
Assume $q \in [1,2)$, $(\mu_n)_{n \in \N} \subset \mathcal P_q(\R^d)$, and $\omega \in \mathcal P_q(\R^d)$. The following statements hold:
\begin{itemize}
    \item[(i)] $\lim_n \mathcal E_q(\mu_n,\omega)= 0$  implies $w-\lim_n \mu_n= \omega$ and $\lim_n \int_{\R^d} |x|^r $ $ d\mu_n(x)$ $= \int_{\R^d} |x|^r d\omega(x)$ for any $r \in (0,q/2)$;
    \item[(ii)]  $w-\lim_n \mu_n= \omega$ and $\sup_n \int_{\R^d} |x|^q d\mu_n(x) < \infty$ implies $\lim_n \mathcal E_q(\mu_n,\omega)= 0$.
\end{itemize}
\end{proposition}
\begin{proof}
    Assume that $\mathcal E_q(\mu_n,\omega) \to 0$ as $n\to \infty$. Then by the Fourier equivalence \eqref{eq:equiv} we have that $\hat \mu_n(\xi)$ converges to $\hat \omega(\xi)$, for a.e. $\xi \in \mathbb R^d$. The equicontinuity of $\hat \mu_n(\xi)$ implies convergence for all $\xi \in \mathbb R^d$. Hence, by Levy continuity theorem \cite[Theorems 14.15 and 18.21]{fristedt1997} we obtain $\mu_n \rightharpoonup \mu$. Moreover  $\mathcal E_q^2(\mu_n,\omega) \to 0$ implies $\sup_n \mathcal E_q^2(\mu_n,\omega) \leq  M < \infty$ and by \cite[Theorem 4.1]{fornasier2014consistency} there exists $M'<\infty$ such that $\sup_n \int_{\R^d} |x|^r d\mu_n(x) \leq M'$ for any $r \in (0,q/2)$. Hence,  \cite[Lemma 5.1.7]{ambrosio2008gradient} yields $\lim_n \int_{\R^d} |x|^r d\mu_n(x)= \int_{\R^d} |x|^r d\omega(x)$ for any $r \in (0,q/2)$. In particular,  $\lim_n W_r(\mu_n,\omega)=0$ for any $r \in (0,q/2)$.\\
    
    Now, assume that $\mu_n \rightharpoonup \omega$ and $\sup_n \int_{\R^d} |x|^q d\mu_n(x) \leq M <\infty$. We would like to show that $\mathcal E_q^2(\mu_n,\omega) \to 0$ as $n\to \infty$. We use again \eqref{eq:equiv} and split the integration as
    \begin{eqnarray*}
     && \int_{\R^d}
  |\hat{\mu}_n(\xi) - \hat{\omega}(\xi)|^2
  \,|\xi|^{-(d+q)}\, d\xi \\
  &=& \underbrace{\int_{|\xi|\leq R}
  |\hat{\mu}_n(\xi) - \hat{\omega}(\xi)|^2
  \,|\xi|^{-(d+q)}\, d\xi}_{A_n(R)} \\
  &+&  \underbrace{\int_{|\xi|> R}
  |\hat{\mu}_n(\xi) - \hat{\omega}(\xi)|^2
  \,|\xi|^{-(d+q)}\, d\xi}_{B_n(R)}.   
    \end{eqnarray*}
    For the second integral we proceed as follows: first of all notice that
    $|\hat\mu(\xi) - 1|^2 = |\hat\mu(\xi)|^2 - 2\,\mathfrak{Re}\,\hat\mu(\xi) + 1$ and from $|\hat\mu(\xi)| \leq 1$ we obtain
    $ |\hat\mu(\xi) - 1|^2 \leq 2(1-\mathfrak{Re}(\hat \mu(\xi)))$. Hence,
    \begin{eqnarray*}
    |\hat{\mu}(\xi) - \hat{\omega}(\xi)|^2 &\lesssim& |\hat{\mu}(\xi) -1 |^2 + |\hat{\omega}(\xi) -1 |^2 \\
    &\lesssim& 2 (1 - \mathfrak{Re}(\hat \mu_n(\xi)) + 1 - \mathfrak{Re}(\hat \omega(\xi))).
    \end{eqnarray*}
    Using the real part of the complex exponential in the Fourier transforms we obtain
    \begin{eqnarray*}
B_n(R) \lesssim \int_{\R^d} \int_{|\xi|> R} \frac{1-\cos(\xi \cdot x)}{|\xi|^{d+q}} d\xi d \mu_n(x) + \int_{\R^d} \int_{|\xi|> R} \frac{1-\cos(\xi \cdot x)}{|\xi|^{d+q}} d\xi d \omega(x).
     \end{eqnarray*}
     We estimate both these integrals in the same manner and we do it exemplarily on first one, by using $1- \cos(t) \leq \min(2,t^2/2)$ and $\min(a,b) \leq a^\theta b^{1-\theta}$
      \begin{eqnarray*}
      && \int_{\R^d} \int_{|\xi|> R} \frac{1-\cos(\xi \cdot x)}{|\xi|^{d+q}} d\xi d \mu_n(x) \\
      &\leq& \int_{\R^d} \int_{|\xi|> R} \frac{\min(2,|\xi|^2|x|^2/2)}{|\xi|^{d+q}} d\xi d \mu_n(x) \\
      &\leq &
      \int_{\R^d} \int_{|\xi|> R} \frac{2^{1-2 \theta}|\xi|^{2\theta} |x|^{2\theta}}{|\xi|^{d+q}} d\xi d \mu_n(x)
       \end{eqnarray*}
       If we choose now $\theta=q/2- \delta$ for $\delta>0$ small and we use spherical coordinates we can further bound the integral as
  \begin{eqnarray*}
      && \int_{\R^d} \int_{|\xi|> R} \frac{1-\cos(\xi \cdot x)}{|\xi|^{d+q}} d\xi d \mu_n(x) \\
&\lesssim&
      \int_{\R^d} \int_{r> R} \frac{r^{q- 2 \delta} |x|^{q-2\delta}}{r^{d+q}} r^{d-1} dr d \mu_n(x)\\
      &=&  \int_{r> R} r^{q-2 \delta} r^{-(d+q)} r^{d-1} dr \int_{\R^d} |x|^{q-2\delta}  d \mu_n(x) \\
      &=& \left( \int_{r> R}  r^{-1-2\delta }dr \right) \left (\int_{\R^d} |x|^{q}  d \mu_n(x) \right)^{(q-2\delta)/q} = \frac{R^{-2\delta}}{2 \delta} \left (\int_{\R^d} |x|^{q}  d \mu_n(x) \right)^{(q-2\delta)/q} \lesssim R^{-2\delta}.
  \end{eqnarray*}     
  Hence, by choosing $R>0$ large enough we can ensure that $B_n(R) \leq \varepsilon/2$, uniformly with respect to $n$. \\
  We turn now to the integral $A_n(R)$, which we split into two terms, one term around $0$ and another in an annulus:
    \begin{eqnarray*}
A_n(R) = \int_{\delta<|\xi|\leq R}
  |\hat{\mu}_n(\xi) - \hat{\omega}(\xi)|^2
  \,|\xi|^{-(d+q)}\, d\xi + \int_{|\xi|\leq \delta}
  |\hat{\mu}_n(\xi) - \hat{\omega}(\xi)|^2
  \,|\xi|^{-(d+q)}\, d\xi.
    \end{eqnarray*} 
    In the first integral  the term $|\xi|^{-(d+q)}$ is bounded on the annulus. As we assumed $\mu_n \rightharpoonup \omega$, we have $\hat \mu_n(\xi) \to \hat \omega(\xi)$ for a.e. $\xi$ and, by dominated convergence, the first integral on the annulus vanishes for $n\to \infty$. For the second integral near the origin, the bound $|e^{it}-1|\leq|t|$ gives
$|\hat\mu_n(\xi)-\hat\omega(\xi)|\leq|\xi|\bigl(\int|x|\,d\mu_n+\int|x|\,d\omega\bigr)$,
so $|\hat\mu_n(\xi)-\hat\omega(\xi)|^2 |\xi|^{-(d+q)}\lesssim|\xi|^{2-(d+q)}$, which is integrable near $0$
since $2-(d+q)>-d$ if and only if $q<2$. This implies
$$
\int_{|\xi|\leq \delta}
  |\hat{\mu}_n(\xi) - \hat{\omega}(\xi)|^2
  \,|\xi|^{-(d+q)}\, d\xi \lesssim \int_{r\leq \delta} r^{2-(d+q)} r^{d-1} dr =O(\delta^{2-q}),
$$
which is made smaller than $\varepsilon/4$ by choosing $\delta$ small.
Hence $A_n(R)<\varepsilon/2$ for all $n$ large enough. We conclude that $\mathcal E_q^2(\mu_n,\omega) \to 0$ as $n\to \infty$.
\end{proof}
\begin{remark}
By \cite[Proposition 7.1.5]{ambrosio2008gradient}, for any $r \geq 1$ it holds that $\lim_n W_r(\mu_n,\omega)=0$ is equivalent to  $\mu_n \rightharpoonup \omega$ for $n\to \infty$ and $\lim_n \int_{\R^d} |x|^r d\mu_n(x) = \int_{\R^d} |x|^r d\omega(x)$. In view of Proposition~\ref{prop:edtop}, convergence in energy distance $\mathcal E_q$ implies convergence in Wasserstein distance $W_r$ for $r \in (0,q/2)$. Conversely, convergence in Wasserstein distance $W_q$ implies convergence in distance $\mathcal E_q$. In Section \ref{sec:W1} we shall prove that 
$\mathcal E_q^2(\mu,\omega)\leq  W_1^{q}(\mu,\omega)\leq W_q^{q}(\mu,\omega)$, which shows that convergence in Wasserstein distance $W_1$ actually suffices to imply convergence in energy distance $\mathcal E_q$. This finer relationship between $\mathcal E_q^2$ and $W_1$ may be surprising as it does not emerge directly from the previous proof, where the uniform boundedness of the $q$-moments is used to estimate integral $B_n(R)$ with a rather natural scaling.  We conclude that energy distances and Wasserstein distance are topologically comparable, but their nature remains intrinsically different. Further comparisons between Wasserstein distances, Fourier-based metrics and energy
distances can be found in
\cite{Auricchio2020,auricchio2026kinetic}.
\end{remark}

\subsection{Discussion and organization of the paper}

Theorem~\ref{t:main-intro} sharpens the asymptotic result of Fornasier and 
H\"utter \cite{fornasier2014consistency} in two distinct ways. On one hand, 
it converts a qualitative narrow-convergence statement
$$
\mu_N^* = \arg\min_{\mu \in \mathcal P^N} \mathcal{E}_q(\mu,\omega) 
\rightharpoonup \omega, \quad N \to +\infty,
$$
into a precise quantitative rate $N^{-\frac{1}{2}(1+q/\expAhlfors)}$. On the 
other hand, by tracking the intrinsic dimension $\expAhlfors$ of the support, 
it explicitly captures the geometric structure of $\omega$.\footnote{We also 
note that the qualitative consistency result by $\Gamma$-convergence in 
\cite{fornasier2014consistency} can actually be obtained more easily as 
follows: in light of the statistical convergence \eqref{eq:stat}, by using 
Markov's inequality one can show that $\min_{\mu \in \mathcal P^N} 
\mathcal{E}^2_q(\mu,\omega) =O(N^{-1})$. Hence, an application of 
Proposition~\ref{prop:edtop}\,(i) implies more directly $\mu_N^* = 
\arg\min_{\mu \in \mathcal P^N} \mathcal{E}_q(\mu,\omega) \rightharpoonup 
\omega$ as $N \to +\infty$.} As pointed out at the beginning of this 
introduction, the rate beats $N^{-1/2}$, showing that the quantization 
minimizers strictly outperform i.i.d.\ Monte Carlo sampling. Along the way 
of proving Theorem~\ref{t:main-intro} we revisit Schoenberg's theory of 
embeddings of measures into homogeneous Sobolev spaces 
(Theorem~\ref{t:equivalenza}), and we establish sharp comparisons between 
the energy distance and the $1$-Wasserstein distance 
(Section~\ref{sec:W1}).

Let us also discuss the relation with the closely related work
\cite{ColzaniGigante}. There, the authors develop a
general theory of discrepancy and numerical integration on metric measure spaces
using stratified random sampling over equal-measure partitions. The probabilistic
idea behind our upper bound is similar in spirit: after decomposing the space into
cells of equal mass, one samples independently one point in each cell and estimates
the resulting error in terms of the cell diameters.

However, our result is not a direct application of that framework. In
\cite{ColzaniGigante}, the functional setting is
intrinsic to the underlying metric measure space: the quadrature error is tested
against potential spaces or Haj{\l}asz--Besov type spaces built from the metric,
the measure, and a suitable kernel. In the present paper, instead, the discrepancy is
fixed a priori by the ambient Euclidean power kernel $-|x-y|^q$, and is identified,
via Schoenberg's theorem, with an ambient negative Sobolev norm on $\mathbb R^d$.
Thus the function space naturally associated with the dual formulation is not
defined from $\omega$ or from the intrinsic geometry of $\mathrm{spt}\,\omega$.
Consequently, even though restrictions of smooth ambient functions to
$\mathrm{spt}\,\omega$ have intrinsic regularity, applying the quadrature
results of \cite{ColzaniGigante} would require a
nontrivial trace or identification theorem between the ambient Sobolev structure and
an intrinsic potential Besov space on $\textup{spt}\,\omega$. For arbitrary
Ahlfors regular supports, possibly fractal or highly irregular, this is not an
automatic matter.

The proof methods are also different. The estimates in
\cite{ColzaniGigante} rely essentially on the
Marcinkiewicz--Zygmund inequality, together with intrinsic kernel estimates. Our
upper bound is more direct in this specific MMD setting: we use the Hilbertian
structure provided by the Schoenberg embedding, so that independence and centering
make the cross terms vanish and the variance is controlled directly by the diameter
of each cell. The lower bound is even further from the quadrature approach. Their
lower estimates provide averaged sharpness results for the stratified quadrature
strategy, namely lower bounds on the mean, over the random choice of one point in
each cell, of the corresponding worst-case integration error. Our lower bound is
deterministic with respect to the empirical configuration: it holds for every choice
of $N$ points and follows from a duality argument in the ambient Sobolev norm,
using Ahlfors-adapted bump functions supported in empty balls. This yields the
matching rate for the MMD quantization problem.

\medskip
The paper is organized as follows. Section~\ref{sec:preliminaries} develops 
the necessary preliminaries on the Fourier transform of measures and the 
homogeneous Sobolev pre-Hilbert structure on $\mathcal{M}_0(\R^d)$, 
culminating in Theorem~\ref{t:equivalenza}, which identifies the energy 
distance with the ${H}^{-s}_0$-norm of the difference of measures. 
Section~\ref{sec:lower bound} establishes the lower bound 
\eqref{eq:Hdots-lower0} via a duality argument with adapted bump functions. 
Section~\ref{sec: upper bound} contains the upper bound 
\eqref{e:upper-bound-main}, proved by a probabilistic equimass-partition 
construction. Finally, Section~\ref{sec:W1} quantitatively establishes the sharp comparison between $\mathcal{E}_q$ and $W_1$.

\section{Preliminaries}\label{sec:preliminaries}
Here and in the following we assume $d\in \N$, $q\in (0,2)$, and set $s:=\frac{d+q}{2}$ unless stated otherwise. We denote by $\mathcal{M} (\mathbb{R}^d)$ the set of signed Radon measures on $\mathbb{R}^d$, and we let
\begin{equation*}
\calM_0(\R^d):=\bigl\{\nu\in\calM(\R^d):\ \nu(\R^d)=0 \textup{ and  $\mathrm{spt}\,\nu$ is compact} \bigr\}
\end{equation*}
be the set of centered signed measures.

\subsection{Fourier Transform}
Given a signed measure $\nu\in\calM(\R^d)$ we define its Fourier transform by
\begin{equation}\label{eq:FT-measure}
\Fhat{\nu}(\xi):=\int_{\R^d} e^{-i\xi\cdot x}\,\dd\nu(x),\qquad \xi\in\R^d.
\end{equation}

\begin{lemma}\label{l:FS-basic}
Let $\nu\in\calM(\R^d)$.
Then:
\begin{enumerate}[label=\arabic*.]
\item $\Fhat{\nu}$ is bounded and $\|\Fhat{\nu}\|_{L^\infty(\R^d)}\le \|\nu\|_{\mathrm{TV}}$;
\item $\Fhat{\nu}$ is uniformly continuous on $\R^d$;
\item $\Fhat{\nu}(0)=\nu(\R^d)$;
\item if $\nu(\R^d) = 0$ then
\begin{equation*}
    |\Fhat{\nu}(\xi)|\le |\xi|\int_{\R^d}|x|\,\dd|\nu|(x).
\end{equation*}
\end{enumerate}
\end{lemma}

\begin{proof}
1. For any $\xi\in\R^d$,
\begin{equation*}
|\Fhat{\nu}(\xi)|
=\left|\int_{\R^d} e^{-i\xi\cdot x}\,\dd\nu(x)\right|
\le \int_{\R^d} |e^{-i\xi\cdot x}|\,\dd|\nu|(x)
=|\nu|(\R^d)=\|\nu\|_{\mathrm{TV}}.    
\end{equation*}

2. Fix $\xi,h\in\R^d$. Then
\begin{equation*}
\Fhat{\nu}(\xi+h)-\Fhat{\nu}(\xi)
=\int_{\R^d} \bigl(e^{-i(\xi+h)\cdot x}-e^{-i\xi\cdot x}\bigr)\,\dd\nu(x),   \end{equation*}
hence, by $|e^{-i\xi \cdot x}|=1$, we obtain
\begin{equation*}
|\Fhat{\nu}(\xi+h)-\Fhat{\nu}(\xi)|
\le \int_{\R^d} |e^{-i(\xi+h)\cdot x}-e^{-i\xi\cdot x}|\,\dd|\nu|(x)=\int_{\R^d} |e^{-ih\cdot x}-1|\,\dd|\nu|(x),
\end{equation*}
where the last integral does not depend on $\xi$. Testing on sequence $h_n\to 0$ and using dominated convergence, we obtain $\int_{\R^d} |e^{-ih\cdot x}-1|\,d|\nu|(x)\to 0$ as $h\to 0$, deducing the claim.

 3. Direct consequence of the definition \eqref{eq:FT-measure} with $\xi=0$.

 4. Since $\nu(\R^d)=0$ we have $\Fhat{\nu}(\xi)=\int (e^{-i\xi\cdot x}-1)\,\dd\nu(x)$.
Using the bound $|e^{-it}-1|\le |t|$ for $t\in\R$, we obtain
\begin{equation*}
|\Fhat{\nu}(\xi)|
\le \int |e^{-i\xi\cdot x}-1|\,\dd|\nu|(x)
\le \int |\xi\cdot x|\,\dd|\nu|(x)
\le |\xi|\int |x|\,\dd|\nu|(x).    
\end{equation*}
\end{proof}

We denote by $\mathcal{S}(\R^d,\mathbb{C})$ the \emph{Schwartz space} of
rapidly decreasing smooth functions, that is
\begin{equation*}
    \mathcal{S}(\R^d,\mathbb{C}) := \Bigl\{ f \in C^\infty(\R^d,\mathbb{C})
    \,:\, \sup_{x\in\R^d} |x^\alpha \partial^\gamma f(x)| < \infty
    \text{ for all multi-indices } \alpha, \gamma \in \N^d \Bigr\}.
\end{equation*}
Equipped with the family of seminorms $\|f\|_{\alpha,\gamma} :=
\sup_{x\in\R^d} |x^\alpha \partial^\gamma f(x)|$,
$\mathcal{S}(\R^d,\mathbb{C})$ is a Fr\'echet space whose continuous dual
$\mathcal{S}'(\R^d,\mathbb{C})$ is the space of tempered distributions.
The Fourier transform is an automorphism of $\mathcal{S}(\R^d,\mathbb{C})$,
and Fourier inversion holds pointwise for every
$f \in \mathcal{S}(\R^d,\mathbb{C})$:
\begin{equation}\label{e:fourier-inversion}
    f(x) = (2\pi)^{-d}\int_{\R^d} e^{i\xi \cdot x}\, \Fhat{f}(\xi)\,\dd\xi,
    \qquad x \in \R^d.
\end{equation}
We shall systematically exploit this
to rewrite pairings $\langle f, \nu \rangle$ between Schwartz functions and
Radon measures as integrals in frequency space, which will allow us to
introduce a natural pre-Hilbert structure on $\calM_0(\R^d)$ via the
homogeneous Sobolev norm $\|\cdot\|_{H^{-s}_0}$,
see Lemma~\ref{l:pairing-formula} and Section~\ref{sec:Hilbert} below.

\begin{lemma}\label{l:pairing-formula}
Let $\nu\in\calM(\R^d)$ and let $f\in\mathcal{S}(\R^d,\mathbb{C})$.
Then
\begin{equation}\label{e:pairing-formula}
\langle f,\nu \rangle:=\int_{\R^d} f(x)\,\dd\nu(x)
=(2\pi)^{-d}\int_{\R^d}\Fhat{f}(\xi)\,\overline{\Fhat{\nu}(\xi)}\,\dd\xi.
\end{equation}
\end{lemma}
\begin{proof}
Integrate both sides of \eqref{e:fourier-inversion} against $\nu$.
Since $\Fhat{f}\in \mathcal{S}(\R^d,\mathbb{C})$, in particular
$\Fhat{f} \in L^1(\R^d)$, so by taking the absolute values and by the finiteness
of $|\nu|$,
\[
\int_{\R^d}\int_{\R^d}\bigl|e^{i\xi\cdot x}\Fhat{f}(\xi)\bigr|\,\dd\xi\,\dd|\nu|(x)
=\|\Fhat{f}\|_{L^1(\R^d)}\,\|\nu\|_{\mathrm{TV}}<\infty.
\]
Fubini's theorem thus applies, yielding
\[
\int_{\R^d} f(x)\,\dd\nu(x)
=(2\pi)^{-d}\int_{\R^d}\Fhat{f}(\xi)
\left(\int_{\R^d}e^{i\xi\cdot x}\,\dd\nu(x)\right)\dd\xi.
\]
Since $\nu$ is real-valued,
$\int e^{i\xi\cdot x}\,\dd\nu(x) = \overline{\Fhat{\nu}(\xi)}$,
which proves \eqref{e:pairing-formula}.
\end{proof}


\begin{corollary}\label{c:uniqueness}
Let $\nu\in\calM(\R^d)$. If $\Fhat{\nu}(\xi)=0$ for all $\xi\in\R^d$, then $\nu=0$.
\end{corollary}

\begin{proof}
Assume $\Fhat{\nu}(\xi) = 0$ for all $\xi \in \R^d$. For every 
$f \in \mathcal{S}(\R^d,\mathbb{C})$, Lemma~\ref{l:pairing-formula} gives
\begin{equation*}
\int_{\R^d} f(x)\,\dd\nu(x)
= (2\pi)^{-d}\int_{\R^d}\Fhat{f}(\xi)\,\overline{\Fhat{\nu}(\xi)}\,\dd\xi = 0.
\end{equation*}
Since $C^\infty_c(\R^d) \subset \mathcal{S}(\R^d,\mathbb{C})$ is dense in 
$C_c(\R^d)$ with respect to the supremum norm, the above implies 
$\int f\,\dd\nu = 0$ for all $f \in C_c(\R^d)$, and the conclusion 
$\nu = 0$ follows from the Riesz representation theorem.
\end{proof}

\subsection{Hilbert Structure on \texorpdfstring{$\calM_0(\R^d)$}{M\_0(R\^{}d)}}\label{sec:Hilbert}
The Fourier-analytic perspective on $\calM(\R^d)$ developed in the 
previous subsection allows us to equip $\calM_0(\mathbb{R}^d)$ with a natural 
pre-Hilbert structure. The key observation is that for $\nu \in \calM_0(\R^d)$ 
the Fourier transform $\Fhat{\nu}$ vanishes at the origin (since 
$\Fhat{\nu}(0) = \nu(\R^d) = 0$), which makes the weight $|\xi|^{-2s}$ 
integrable near zero against $|\Fhat{\nu}|^2$, as we verify in 
Proposition~\ref{p:finitezza norma} below. This motivates us to introduce the following negative Sobolev norm.
For every $\nu\in \calM(\R^d)$ and $f\in \mathcal{S}(\R^d,\mathbb{C})$ 
we define the following
\begin{equation}\label{e:norma sobolev negativa}
    \|\nu\|_{H_0^{-s}}:=\int_{\R^d}|\Fhat{\nu}(\xi)|^2\,|\xi|^{-2s}\,\dd \xi \quad \textup{and} \quad  \|f\|_{H_0^{s}}:=\int_{\R^d}|\Fhat{f}(\xi)|^2\,|\xi|^{2s}\,\dd \xi.
\end{equation}
Note that $0 < \|f\|_{H_0^{s}} < \infty$ for every non-zero 
$f \in \mathcal{S}(\R^d,\mathbb{C})$, since Schwartz functions decay 
faster than any polynomial and their Fourier transforms do likewise, 
so the weight $|\xi|^{2s}$ does not affect finiteness of the integral.
\begin{proposition}\label{p:finitezza norma}
Let $\nu\in\calM(\R^d)$, then 
\begin{itemize}
    \item[1.] if $\nu(\R^d)\neq 0$ then $\|\nu\|_{H_0^{-s}}=\infty$;
    \item[2.] if $\nu\in \calM_0(\R^d)$ then $\|\nu\|_{H_0^{-s}}<\infty$.
\end{itemize}
\end{proposition}

\begin{proof}
1. By Lemma~\ref{l:FS-basic} we have that $\Fhat{\nu}(0)=\nu(\R^d)\neq 0$ and $\Fhat{\nu}$ is continuous. Therefore there exist $\delta>0$ and $c_0>0$ such that
$|\Fhat{\nu}(\xi)|\ge c_0$ for all $|\xi|\le \delta$. In particular
\begin{align*}
\int_{\R^d}|\Fhat{\nu}(\xi)|^2\,|\xi|^{-(d+q)}\,\dd\xi&
 \ge\ \int_{|\xi|\le\delta} c_0^2\,|\xi|^{-(d+q)}\,\dd\xi\\
 &=c_0^2\int_0^\delta r^{-(d+q)}\,r^{d-1}\,\dd r=c_0^2\int_0^\delta r^{-1-q}\,\dd r=+\infty.    
\end{align*}

2. We split the integral:
\begin{equation*}
\int_{\R^d}|\Fhat{\nu}(\xi)|^2|\xi|^{-(d+q)}\,\dd\xi
=\int_{|\xi|\le 1}|\Fhat{\nu}(\xi)|^2|\xi|^{-(d+q)}\,\dd\xi+\int_{|\xi|\ge 1}|\Fhat{\nu}(\xi)|^2|\xi|^{-(d+q)}\,\dd\xi.
\end{equation*}

By Lemma~\ref{l:FS-basic} item 1., $|\Fhat{\nu}(\xi)|\le \|\nu\|_{\mathrm{TV}}$. Hence
\begin{equation*}
\int_{|\xi|\ge 1}|\Fhat{\nu}(\xi)|^2|\xi|^{-(d+q)}\,\dd\xi
\le \|\nu\|_{\mathrm{TV}}^2 \int_{|\xi|\ge 1}|\xi|^{-(d+q)}\,\dd\xi=\|\nu\|_{\mathrm{TV}}^2\int_1^\infty r^{-1-q}\,\dd r<\infty.
\end{equation*}

Since $\nu\in\calM_0(\R^d)$, Lemma~\ref{l:FS-basic} item 4. yields
$|\Fhat{\nu}(\xi)|\le C_\nu|\xi|$ for all $\xi$, where $C_\nu=\int_\Omega |x|\,d|\nu|(x)<\infty$, because $\nu$ has compact support.
Therefore
\begin{equation*}
\int_{|\xi|\le 1}|\Fhat{\nu}(\xi)|^2|\xi|^{-(d+q)}\,\dd\xi
\le C_\nu^2\int_{|\xi|\le 1}|\xi|^{2-(d+q)}\,\dd\xi= C^2_{\nu}\int_0^1 r^{1-q}\,\dd r<\infty.
\end{equation*}
Combining the two parts proves finiteness.
\end{proof}
Proposition~\ref{p:finitezza norma} shows that $\|\cdot\|_{H_0^{-s}}$ 
is a well-defined finite quantity on $\calM_0(\R^d)$. To upgrade this 
to a genuine pre-Hilbert structure, we introduce the following inner 
product on $\calM_0(\R^d)$: for every $\nu, \eta \in \calM_0(\R^d)$ define
\begin{equation}\label{e:innerprod}
\langle \nu,\eta\rangle_{H_0^{-s}}
:=\int_{\R^d}\mathfrak{Re}\big(\frac{\Fhat{\nu}(\xi)\,\overline{\Fhat{\eta}(\xi)}+\overline{\Fhat{\nu}(\xi)}\,{\Fhat{\eta}(\xi)}}{2}\big)\,|\xi|^{-2s}\,\dd\xi.
\end{equation}

\begin{proposition}
 $(\calM_0(\R^d), \langle \cdot , \cdot\rangle_{H^{-s}_0})$ is a real Pre-Hilbert space.   
\end{proposition}

\begin{proof}
 For every $\nu,\eta\in\calM_0(\R^d)$ we have that \eqref{e:innerprod} is well-defined, indeed
 \begin{align*}
    \int_{\R^d}\bigg|\mathfrak{Re}\bigg(\frac{\Fhat{\nu}(\xi)\,\overline{\Fhat{\eta}(\xi)}+\overline{\Fhat{\nu}(\xi)}\,{\Fhat{\eta}(\xi)}}{2}\bigg)\bigg|\,|\xi|^{-2s}\,\dd\xi&\leq  \int_{\R^d}|\Fhat{\nu}(\xi)\,\overline{\Fhat{\eta}(\xi)}|\,|\xi|^{-2s}\,\dd\xi\\ 
    &= \int_{\R^d}|\Fhat{\nu}(\xi)| |\xi|^{-s}\,|{\Fhat{\eta}(\xi)}|\,|\xi|^{-s}\,\dd\xi 
 \end{align*}
 and by Proposition~\ref{p:finitezza norma} item 2. we have that $|\Fhat{\nu}(\xi)| |\xi|^{-s}, |\Fhat{\eta}(\xi)| |\xi|^{-s} \in L^2(\R^d)$, therefore we can use Cauchy-Schwartz inequality in $L^2(\R^d)$ to infer
 \begin{equation*}
          \int_{\R^d}\big|\mathfrak{Re}\big(\frac{\Fhat{\nu}(\xi)\,\overline{\Fhat{\eta}(\xi)}+\overline{\Fhat{\nu}(\xi)}\,{\Fhat{\eta}(\xi)}}{2}\big)\big|\,|\xi|^{-2s}\,\dd\xi \leq \|\nu\|_{H^{-s}_0}\|\eta\|_{H^{-s}_0}<\infty.
 \end{equation*}

 The $\R$-bilinearity and the symmetry of $\langle \cdot , \cdot\rangle_{H^{-s}_0}$ are clear from definition \eqref{e:innerprod}, from which we can also deduce that $\langle \nu,\nu\rangle_{H^{-s}_0}^{\frac12}=\|\nu\|_{H^{-s}_0}$ for every $\nu\in \calM_0(\R^d)$.

 To conclude the proof we have to show that $\|\nu\|_{H^{-s}_0}=0$ if and only if $\nu=0$. An implication is clear. Assume now $\|\nu\|_{H^{-s}_0}=0$, therefore we have $\Fhat{\nu}(\xi)=0$ for $\mathcal{L}^d$-a.e. $\xi\in \R^d$ and, being $\Fhat{\nu}$ continuous by Lemma~\ref{l:FS-basic} item 2., we obtain that $\Fhat{\nu}(\xi)=0$ for every $\xi\in\R^d$. The conclusion follows from Corollary \ref{c:uniqueness}. 
\end{proof}

\begin{remark}
 Let $\overline{\calM_0}$ be the completion of $\calM_0(\R^d)$ with respect to the metric induced by the norm
$\|\cdot\|_{H_0^{-s}}$ and consider the $\R$-linear isometry $T:\calM_0(\R^d)\to L^2(\R^d,\mathbb{C})$ given by 
\begin{equation*}
    T(\nu)(\xi):= \Fhat{\nu}(\xi)|\xi|^{-s}.
\end{equation*}
In particular, being $L^2(\R^d,\mathbb{C})$ separable, we have that $(\overline{\calM_0},\langle\cdot,\cdot\rangle_{H^{-s}_0})$ is a separable Hilbert space.
\end{remark}
The following density result will be used to identify the dual norm 
on $\calM_0(\R^d)$ in Proposition~\ref{p:dual-norm-exact} below, 
which is the key tool in the proof of the lower bound in 
Section~\ref{sec:lower bound}.
\begin{lemma}\label{l:dense-weighted-Schwartz}
The set $\mathcal D_s
:=\bigl\{\,|\xi|^{s}\Fhat{f}(\xi): f\in\mathcal S(\R^d)\,\bigr\}$ is dense in $L^2(\R^d,\mathbb{C})$.
\end{lemma}

\begin{proof}
Since the Fourier transform is an automorphism of $\mathcal S(\R^d,\mathbb{C})$, the following equality holds $\mathcal D_s=\{\,|\xi|^{s}g(\xi): g\in\mathcal S(\R^d,\mathbb{C})\,\}.
$ We claim that $C_c^\infty(\R^d\setminus\{0\})\subset \mathcal D_s$.
Indeed, if $\phi\in C_c^\infty(\R^d\setminus\{0\},\mathbb{C})$ then $g(\xi):=|\xi|^{-s}\phi(\xi)$ is smooth with compact support
away from $0$, hence $g\in\mathcal S(\R^d,\mathbb{C})$ and $\phi=|\xi|^{s}g\in\mathcal D_s$.

Finally, $C_c^\infty(\R^d\setminus\{0\},\mathbb{C})$ is dense in $L^2(\R^d,\mathbb{C})$:
given $h\in L^2(\R^d,\mathbb{C})$, first approximate $h$ by a compactly supported smooth function $\tilde h\in C_c^\infty(\R^d,\mathbb{C})$,
then multiply $\tilde h$ by a smooth cutoff that vanishes on $B(0,\varepsilon)$ and equals $1$ outside $B(0,2\varepsilon)$.
Letting $\varepsilon\downarrow 0$ gives convergence in $L^2(\R^d,\mathbb{C})$ and hence $\overline{\mathcal D_s}^{\,L^2}=L^2(\R^d,\mathbb{C})$.
\end{proof}

\begin{proposition}\label{p:dual-norm-exact}
Let $\nu\in {\calM_0}(\R^d)$, then
\begin{equation}\label{eq:dual-norm-exact}
\sup\Bigl\{\,|\langle f,\nu\rangle|:\ f\in\mathcal S(\R^d,\mathbb{C}),\ \|f\|_{H_0^{s}(\R^d)}\le 1\,\Bigr\}
=(2\pi)^{-d}\,\|\nu\|_{H^{-s}_0}.
\end{equation}
where $\langle f,\nu \rangle$ is defined in \eqref{e:pairing-formula}.
\end{proposition}

\begin{proof}
From Lemma~\ref{l:pairing-formula} we have that
\begin{equation}\label{e:pairing due}
\langle f,\nu\rangle
=(2\pi)^{-d}\,\langle\, \Fhat{\nu}|\xi|^{-s},\Fhat f|\xi|^{s}\,\rangle_{L^2(\R^d,\mathbb{C})},
\end{equation}
therefore by Cauchy--Schwarz in $L^2(\R^d,\mathbb{C})$ we have
\begin{equation}\label{Cauchy-Schwartz sob}
|\langle f,\nu\rangle|
\le (2\pi)^{-d}\,\|f\|_{H_0^{s}}\,\|\nu\|_{H^{-s}_0}.
\end{equation}
Taking the supremum over $\|f\|_{H_0^{s}}\le 1$ gives
\begin{equation*}
\sup_{\|f\|_{\dot H^{s}}\le 1}|\langle f,\nu\rangle|\le (2\pi)^{-d}\,\|\nu\|_{H^{-s}_0}.
\end{equation*}

In case $\nu=0$ then the claim is obvious, otherwise let $h(\xi)={\|\nu\|^{-1}_{H^{-s}_0}}\Fhat{\nu}(\xi)|\xi|^{-s}$ and from Lemma~\ref{l:dense-weighted-Schwartz}, we can choose $f_k\in\mathcal{S}(\R^d,\mathbb{C})$ such that
$|\xi|^{s}\Fhat f_k\to h$ in $L^2(\R^d,\mathbb{C})$ as $k\to\infty$.
Then $\|f_k\|_{H_0^{s}}=\||\xi|^{s}\Fhat f_k\|_{L^2(\R^d,\mathbb{C})}\to\|h\|_{L^2(\R^d,\mathbb{C})}=1$.
Define $g_k:=f_k/\|f_k\|_{H_0^{s}}\in \mathcal{S}(\R^d,\mathbb{C})$, so that $\|g_k\|_{H_0^{s}}=1$ and
$|\xi|^{s}\Fhat g_k\to h$ in $L^2(\R^d,\mathbb{C})$.
Using \eqref{e:pairing due} we get
\begin{align*}
|\langle g_k,\nu\rangle|
&= (2\pi)^{-d}\|\nu\|_{H^{-s}_0}\,
   \bigl|\langle h,\,|\xi|^{s}\Fhat{g}_k\rangle_{L^2(\R^d,\mathbb{C})}\bigr|\\
&\longrightarrow\ 
   (2\pi)^{-d}\|\nu\|_{H^{-s}_0}\,
   |\langle h,h\rangle_{L^2}|
= (2\pi)^{-d}\|\nu\|_{H^{-s}_0}.
\end{align*}
\end{proof}

\begin{corollary}\label{c:Schwartz-separates}
Let $\nu\in \overline{\calM_0}$ and assume
\begin{equation*}
    \langle f,\nu\rangle=0
\qquad\text{for all }f\in\mathcal S(\R^d,\mathbb{C}),
\end{equation*}
then $\nu=0$.
\end{corollary}

\begin{proof}
It is a direct consequence of \eqref{e:pairing due}, Lemma~\ref{l:dense-weighted-Schwartz} and Riesz's Theorem in $L^2(\R^d,\mathbb{C})$.
\end{proof}
We conclude this section with the central equivalence result that 
underpins all that follows: the ${H}^{-s}_0$-norm on $\calM_0(\R^d)$ 
coincides, up to a explicit constant, with the energy distance 
$\mathcal{E}_q$, identifying the pre-Hilbert structure just constructed 
with the MMD with power kernel introduced in Section~\ref{sec:intro_MMD}.
\begin{theorem}\label{t:equivalenza}
For every $\nu \in \calM_0(\R^d)$, the following holds
\begin{equation}\label{e:equiv}
\|\nu\|_{H_0^{-s}}^2
= -\,\frac{\kappa_{d,q}}{2}\iint_{\R^d\times\R^d}|x-y|^q\,d\nu(x)\,d\nu(y),
\end{equation}
where $\kappa_{d,q}$ is the constant appears in \ref{t:main-intro}. In particular the map $\Phi:\R^d\to \mathcal{M}_0(\R^d)$ defined by $\Phi(x)=\delta_x-\delta_0$, satisfies
\begin{equation}\label{e:isometry}
   \|\delta_x -\delta_{y}\|^2_{H^{-s}_0}= \|\Phi(x) -\Phi(y)\|^2_{H^{-s}_0}= \kappa_{d,q}|x-y|^q
\end{equation}
for every $x,y\in \R^d$.
\end{theorem}

\begin{proof}
The statement is a particular case of the work of Sch\"onberg \cite{schoenberg1938positive}. The equivalence \eqref{e:equiv} between the energy distance
and its Fourier representation
appeared multiple times in the literature in the course of the years. A first address is the book by Berg--Christensen--Ressel~\cite[Chapter 3, Section 2, Corollary 2.10]{BergChristensenRessel1984}.
It has been also informally derived in
Sz\'ekely's 1989 lecture notes~\cite{szekely1989} with an  explicit proof
(for $d=1$) is in the 2002 BGSU technical report~\cite{szekely2002}, followed by a peer-reviewed proof for arbitrary $d$ and $q=1$ is in~\cite{szekely2005}.
The full result for $0<q<2$ and any $d\geq 1$, together with the explicit
constant $\kappa_{d,q}$, appears again as Proposition~1 of~\cite{szekely2013energy}.
This result has been redescovered independently by Fornasier-Hütter in \cite[Theorem 3.6]{fornasier2014consistency} using the theory of conditionally positive semidefinite functions and the generalized Fourier transform following the book \cite{wendland2005}. One more proof of this result can be found in the paper \cite[Theorem 1]{auricchio2026kinetic}.
\end{proof}

\section{Quantitative Rate of Convergence of Empirical Measures}
We fix now $\omega\in \mathcal{P}(\R^d)$ such that there exist $C_\omega > 1,R_0>0$ and $\expAhlfors\in (0,d]$ with
\begin{equation}\tag{Ahl}\label{e:alfhors reg}
C^{-1}_\omega r^\expAhlfors \leq \omega(B_r(x)) \leq C_\omega r^\expAhlfors 
\end{equation}
for every $x\in \mathrm{spt}\, \omega$ and every $r\in (0,R_0)$. 
Measures satisfying \eqref{e:alfhors reg} are termed Ahlfors regular
in the literature. In particular, standard density estimates (cf. \cite[Theorem 2.56]{AFP}) imply that for every Borel set $B$ in 
$\mathrm{spt}\, \omega$
\[
C^{-1}_\omega \mathcal{H}^\expAhlfors(B)\leq\omega(B)\leq
2^\expAhlfors C_\omega \mathcal{H}^\expAhlfors(B)\,,
\]
where $C_\omega$ is the constant appearing in \eqref{e:alfhors reg}.
If $\expAhlfors\in\N\cap(0,d]$, then $\mathcal{H}^\expAhlfors\res E$ obeys \eqref{e:alfhors reg} provided that $E$ is $\expAhlfors$-rectifiable.
For general exponents $\beta$ as above, self-similar sets satisfying the open set condition according to \cite{Hutchinson81}, such as the Cantor set in $\R$ or the Sierpinski gasket, support Ahlfors regular measures.

\begin{remark}\label{r:supporto compatto}
Property \eqref{e:alfhors reg} implies that $\textup{spt}\, \omega$ must be compact. Indeed, if the support were not totally bounded, we could find a radius $\rho \in (0,R_0)$ and an infinite sequence of points $\{x_j\}_{j\in \N} \subset \textup{spt}\, \omega$ satisfying
\begin{equation*}
    \|x_i - x_j\| \geq 2\rho \quad \text{for all } i \neq j.
\end{equation*}
This would imply that the open balls $\{B_{\rho}(x_j)\}_{j\in \N}$ are pairwise disjoint. As a result, we would obtain
\begin{equation*}
    1 = \omega(\R^d) \geq \omega\left(\bigcup_{j=1}^\infty B_{\rho}(x_j)\right) = \sum_{j=1}^\infty \omega(B_{\rho}(x_j)) \geq C^{-1}_\omega\sum_{j=1}^\infty \rho^\expAhlfors = \infty,
\end{equation*}
leading to a contradiction.
\end{remark}

\subsection{Lower bound}\label{sec:lower bound}
The aim of this section is to show the lower bound in Theorem\ref{t:main-intro}. We stress out that no connected assumption on $\textup{spt} \; \omega$ is needed here. We argue by duality of $\mathcal{M}_0(\mathbb{R}^d)$ with $\mathcal{S}(\R^d,\mathbb{C})$ (cf. Proposition \ref{p:dual-norm-exact}): we test $\mu_N-\omega$ with a sum of bumps functions that vanishes on $\mu_N$, keeping the $H^s_0$-norm controlled by scaling and separation.

\begin{lemma}\label{l:palle vuote}
Assume condition \eqref{e:alfhors reg}. Then there exist constants $\alpha\in(0,\infty)$ and $N_0\in\N$
(depending only on $n,C,R_0$) with the following property: for every $N\ge N_0$ and every sequence of points $\{x_j\}_{j=1}^N$ in $\R^d$, there exists a sequence of points $\{z_j\}_{j=1}^N$ in  
$\mathrm{spt}\,\omega$ 
such that the balls $\{B_{\rho_N}(z_j)\}_{j=1}^N$ are pairwise disjoint and 
\begin{equation}\label{e:palle vuote}
B_{\rho_N}(z_j)\cap\{x_1,\dots,x_N\}=\emptyset \qquad\forall j=1,\dots,m,
\end{equation}
where $\rho_N:=\alpha N^{-\frac1\expAhlfors}$.
\end{lemma}

\begin{proof}
Fix $N$ large enough so that
\begin{equation}\label{eq:r-choice}
r_N:= (2C_\omega N)^{-\frac1\expAhlfors} \leq \frac{R_0}{2}.
\end{equation}
This is possible for all $N\ge N_0$ with $N_0$ depending only on $C_\omega,R_0$. Let $Z=\{z_1,\dots,z_M\}\subset\mathrm{spt}\,\omega$ be a maximal $r$-separated set, i.e.
$|z_j-z_k|\ge r_N$ for $j\neq k$, and maximal under inclusion. We notice that $Z$ must be finite because of assumption \eqref{e:alfhors reg}, $\omega \in \mathcal{P}(\R^d)$ and the balls $\{B_{\frac{r_N}{2}}(z_j)\}_{j=1}^M$ are pairwise disjoint.

Maximality condition implies the covering $\textup{spt} \; \omega\subset\bigcup_{j=1}^M B_{r_N}(z_j)$:
indeed, if there were $x\in\textup{spt} \;\omega$ with $|x-z_j|\ge r$ for all $j$, then $Z\cup\{x\}$ would still be $r$-separated,
contradicting maximality.

Using this cover and assumption \eqref{e:alfhors reg}, we obtain
\begin{equation*}
1=\omega(\textup{spt} \;\omega)\le \sum_{j=1}^M \omega(B_{r_N}(z_j))\le M\,C_\omega r_N^\expAhlfors,    
\end{equation*}
thanks to \eqref{eq:r-choice}.
Thus $M\geq 2N$. 

Consider the disjoint balls $\{B_{\frac{r_N}{2}}(z_j)\}_{j=1}^M$.
Since they are disjoint, each point $x_j$ can belong to at most one such ball,
hence at most $N$ of these balls contain at least one sample point $x_j$.
Because $M\ge 2N$, at least $M-N\ge N$ balls are empty of sample points.
Relabel $N$ of these centers as $z_1,\dots,z_N$, and upon setting
$\rho_N := \frac{r_N}{2} $ 
we obtain \eqref{e:palle vuote} with $\alpha:=\frac12(2C_\omega)^{-\frac1\expAhlfors}$.
\end{proof}

Fix a function $\psi\in C_c^\infty(B_1(0))$ such that
\begin{equation}\label{e:psi-choice}
0\le \psi\le 1,\qquad \psi= 1\ \text{on }B_{\frac12}(0).
\end{equation}
For $r>0$ and $z\in\R^d$ define the rescaled bump
\begin{equation}\label{e:psi riscalata}
\psi_{r,z}(x):=r^{\,s-\frac d2}\,\psi\!\left(\frac{x-z}{r}\right)=r^{\frac{q}{2}}\,\psi\!\left(\frac{x-z}{r}\right),\qquad x\in\R^d.
\end{equation}

\begin{lemma}\label{l:Hdots-scaling}
For every $r>0$ and $z\in\R^d$ one has
\begin{equation*}
\|\psi_{r,z}\|_{H_0^s}=\|\psi\|_{H_0^s}.
\end{equation*}
\end{lemma}

\begin{proof}
We compute the Fourier transform of $\psi_{r,z}$.
Using the change of variables $y=(x-z)/r$, we have
\begin{equation}\label{e:fourier per riscalata}
\Fhat{\psi_{r,z}}(\xi)
=\int_{\R^d} e^{-i\xi\cdot x}\,r^{\,s-\frac d2}\psi \left(\frac{x-z}{r}\right)\,\dd x
=r^{\,s-\frac d2}r^d \int_{\R^d} e^{-i\xi\cdot(z+r y)}\,\psi(y)\,\dd y
=r^{\,s+\frac d2}\,e^{-i\xi\cdot z}\,\Fhat{\psi}(r\xi).
\end{equation}
Hence
\begin{equation*}
\|\psi_{r,z}\|_{H_0^s}^2
=\int_{\R^d} |\xi|^{2s}\,|\Fhat{\psi_{r,z}}(\xi)|^2\,\dd\xi
=\int_{\R^d} |\xi|^{2s}\,r^{2s+d}\,|\Fhat{\psi}(r\xi)|^2\,\dd\xi.
\end{equation*}
Change variables $\eta=r\xi$, so $d\xi=r^{-d}d\eta$ and $|\xi|^{2s}=r^{-2s}|\eta|^{2s}$:
\begin{equation*}
\|\psi_{r,z}\|_{H_0^s}^2
=\int_{\R^d} r^{-2s}|\eta|^{2s}\,r^{2s+d}\,|\Fhat{\psi}(\eta)|^2\,r^{-d}\,\dd\eta
=\int_{\R^d} |\eta|^{2s}|\Fhat{\psi}(\eta)|^2\,\dd\eta
=\|\psi\|_{H_0^s}^2.
\end{equation*}

\end{proof}

\begin{lemma}\label{l:conteggio}
Let $\delta\in (0,4]$ and let $Y\subset\R^d$ be $\delta$-separated, i.e. $
|y-y'|\ge \delta \text{ for all distinct }y,y'\in Y$.
Then for every $R>0$,
\begin{equation*}
\#(Y\cap B_R(0)) \le 2^d\left(\frac{R}{\delta}+1\right)^d.   \end{equation*}
Moreover, there exists a constant $C_d>0$, depending only on $d$, such that for every integer $m\geq 3$,
\begin{equation*}
\#\bigl(Y\cap\{y\in\R^d:m\le |y|<m+1\}\bigr)
\le C_d\,\delta^{-d}\,m^{d-1}.    
\end{equation*}
\end{lemma}

\begin{proof}
Consider the family of open balls $\{B_{\frac{\delta}{2}}(y)\}_{y\in Y}$. Since $Y$ is $\delta$-separated, these balls are pairwise disjoint. Now letting $R>0$ we have that $B_{\frac{\delta}{2}}(y)\subset B_{R+\delta/2}(0)$ and therefore
\begin{equation*}
\#(Y\cap B_R(0))\,\mathcal L^d(B_{\frac{\delta}{2}}(0))
\le \mathcal L^d(B_{R+\frac{\delta}{2}}(0)).    
\end{equation*}
Setting $\omega_d:=\mathcal L^d(B_1(0))$, we conclude
\begin{equation*}
    \#(Y\cap B_R(0))
\le \frac{\omega_d(R+\delta/2)^d}{\omega_d(\delta/2)^d}
= \left(\frac{2R}{\delta}+1\right)^d\leq 2^d\left(\frac{R}{\delta}+1\right)^d.
\end{equation*}

Fix now an integer $m\ge 3$ and set
\begin{equation*}
    A_m:=Y\cap\{y\in\R^d:m\le |y|<m+1\}.
\end{equation*}
Let $y\in A_n$ and  $x\in B_{\frac{\delta}{2}}(y)$. Since $|y|<m+1$, we have
\begin{equation*}
|x|\le |x-y|+|y|<\delta/2+m+1,    
\end{equation*}
while since $|y|\ge m$, we also have
\begin{equation*}
|x|\ge |y|-|x-y|>m-\delta/2.
\end{equation*}
Therefore
\begin{equation*}
B_{\frac{\delta}{2}}(y)\subseteq B_{m+1+\delta/2}(0)\setminus B_{m-\delta/2}(0).
\end{equation*}
Since the balls $\{B_{\frac{\delta}{2}}(y)\}_{y\in A_m}$ are pairwise disjoint, we obtain
\begin{equation*}
\#(A_m)\,\mathcal L^d(B_{\frac{\delta}{2}}(0))
\le
\mathcal L^d\bigl(B_{m+1+\delta/2}(0)\setminus B_{m-\delta/2}(0)\bigr).
\end{equation*}
Hence
\begin{equation*}
\#(A_m)
\le
\frac{\omega_d\bigl((m+1+\delta/2)^d-(m-\delta/2)^d\bigr)}
{\omega_d(\delta/2)^d}
=
\frac{(m+1+\delta/2)^d-(m-\delta/2)^d}{(\delta/2)^d}.    
\end{equation*}
By the mean value theorem, there exists
$\xi_m\in (\,m-\delta/2,\; m+1+\delta/2\,)$
such that
\begin{equation*}
(m+1+\delta/2)^d-(m-\delta/2)^d
=
d\,\xi_m^{\,d-1}(1+\delta)\leq 5d\,\xi_m^{\,d-1} .    
\end{equation*}
Since $m\ge 3$, we have
\begin{equation*}
\xi_m\le m+1+\delta/2 \le 2\, m
\end{equation*}
and therefore
\begin{equation*}
\#(A_m)\le C_d\,\delta^{-d}\,m^{d-1},
\end{equation*}
where $C_d:=5d \cdot4^d $.
\end{proof}

\begin{lemma}\label{l:almost-orth}
There exists a constant $\Fhat{C}$, depending only on $\psi$, $d$ and $s$, such that for every $N\in \N$, $r>0$, and $z_1,\dots,z_N\in\R^d$ satisfy the separation condition
\begin{equation}\label{e:separation}
|z_j-z_k|\ge 4r\qquad\text{for all }j\neq k,
\end{equation}
we have that
\begin{equation}\label{eq:almost-orth-bound}
\left\|\sum_{j=1}^N \psi_{r,z_j}\right\|_{H_0^s}^2 \le \Fhat{C}\,N.
\end{equation}
\end{lemma}

\begin{proof}
By polarization,
\begin{equation}\label{e:pitagora}
\left\|\sum_{j=1}^N \psi_{r,z_j}\right\|_{H_0^s}^2
=
\sum_{j=1}^N \|\psi_{r,z_j}\|_{H_0^s}^2
+
2\sum_{1\le j<k\le N}
\langle \psi_{r,z_j},\psi_{r,z_k}\rangle_{H_0^s},  \end{equation}
where
\begin{equation*}
\langle g,f\rangle_{H_0^s}
:=
\int_{\R^d} |\xi|^{2s}\,\widehat f(\xi)\,\overline{\widehat g(\xi)}\,\dd\xi.
\end{equation*}
By Lemma~\ref{l:Hdots-scaling} we have $\|\psi_{r,z_j}\|_{H_0^s}=\|\psi\|_{H_0^s}$  for every $j=1,\dots,N$ and hence
\begin{equation*}
    \sum_{j=1}^N \|\psi_{r,z_j}\|_{H_0^s}^2
=
N\,\|\psi\|_{H_0^s}^2.
\end{equation*}

From \eqref{e:fourier per riscalata} we have
\begin{equation*}
\widehat{\psi_{r,z}}(\xi)
=
r^{s+\frac d2} e^{-i\xi\cdot z}\widehat\psi(r\xi),    
\end{equation*}
and therefore for $j\neq k$ we obtain:
\begin{align*}
\langle \psi_{r,z_j},\psi_{r,z_k}\rangle_{H_0^s}
&=
\int_{\R^d}
|\xi|^{2s}\,
r^{2s+d}\,
e^{-i\xi\cdot(z_j-z_k)}
|\widehat\psi(r\xi)|^2\,\dd\xi.
\end{align*}
Changing variables $\eta=r\xi$, we get
\begin{equation*}
\langle \psi_{r,z_j},\psi_{r,z_k}\rangle_{H_0^s}
=
\int_{\R^d}
|\eta|^{2s}\,
e^{-i\eta\cdot \frac{z_j-z_k}{r}}
|\widehat\psi(\eta)|^2\,\dd\eta.    
\end{equation*}
Setting
\begin{equation}\label{eq:Gamma-def}
\Gamma(y)
:=
\int_{\R^d}
|\eta|^{2s}|\widehat\psi(\eta)|^2 e^{-i\eta\cdot y}\,\dd\eta,
\qquad y\in\R^d,
\end{equation}
we get
\begin{equation}\label{eq:cross-term-Gamma}
\langle \psi_{r,z_j},\psi_{r,z_k}\rangle_{H_0^s}
=
\Gamma \left(\frac{z_j-z_k}{r}\right).
\end{equation}

Since $\psi\in C_c^\infty(\R^d)$, then $\eta\mapsto |\eta|^{2s}|\widehat\psi(\eta)|^2$ is in $\mathcal{S}(\R^d,\mathbb{C})$ and in particular also $\Gamma \in \mathcal S(\R^d,\mathbb{C})$. Therefore, for every $M\in \N$ there exists a constant $C_M<\infty$ such that
\begin{equation}\label{eq:Gamma-decay}
|\Gamma(y)|\le C_M(1+|y|)^{-M}\qquad\forall y\in\R^d.
\end{equation}

Fix now $j\in\{1,\dots,N\}$ and define
\begin{equation*}
Y_j:=\left\{\frac{z_j-z_k}{r}:k\neq j\right\}\subset\R^d.    
\end{equation*}
We claim that $Y_j$ is $4$-separated. Indeed, if $k\neq k'$ and both are different from $j$, then
\begin{equation*}
\left|\frac{z_j-z_k}{r}-\frac{z_j-z_{k'}}{r}
\right|
=
\frac{|z_k-z_{k'}|}{r}
\ge 4    
\end{equation*}
by \eqref{e:separation}. 

Choose $M>d$. Then by \eqref{eq:Gamma-decay} and Lemma~\ref{l:conteggio} (applied with $\delta=4$),
\begin{align*}
  &  \sum_{k\neq j}
\left|
\Gamma\left(\frac{z_j-z_k}{r}\right)
\right|
=
\sum_{y\in Y_j} |\Gamma(y)|=\sum_{n=0}^\infty
\sum_{y\in Y_j\cap\{m\le |y|<m+1\}} |\Gamma(y)| \\ & \leq  C_M\#\bigl(Y_j\cap B_3(0)\bigr)+  C_M
\sum_{m=3}^\infty
\#\bigl(Y_j\cap\{m\le |y|<m+1\}\bigr)(1+m)^{-M} \\ & \leq 4^d+ C_MC_d \sum_{m=3} \frac{m^{d-1}}{(1+m)^M}<\infty.
\end{align*}
 Hence there exists a constant $K$, that depens only on $\psi$, $d$ and $s$, such that for every $j=1,\dots,N$
\begin{equation*}
\sum_{k\neq j}^N
\left|
\langle \psi_{r,z_j},\psi_{r,z_k}\rangle_{H_0^s}
\right|
\le K.
\end{equation*}

In particular
\begin{equation*}
\sum_{1\le j<k\le N}
\left|
\langle \psi_{r,z_j},\psi_{r,z_k}\rangle_{H_0^s}
\right|
=
\frac12
\sum_{j=1}^N
\sum_{k\neq j}
\left|
\langle \psi_{r,z_j},\psi_{r,z_k}\rangle_{H_0^s}
\right|
\le
\frac12\,N\,K.    
\end{equation*}
Combining this with \eqref{e:pitagora} we obtain
\begin{equation*}
\left\|\sum_{j=1}^N \psi_{r,z_j}\right\|_{H_0^s}^2
\le
N\|\psi\|_{H_0^s}^2 + 2\cdot \frac12\,N\,K
=
N\bigl(\|\psi\|_{H_0^s}^2 + K\bigr).    
\end{equation*}
Therefore \eqref{eq:almost-orth-bound} holds with
\begin{equation*}
\Fhat{C}:=\|\psi\|_{H_0^s}^2 + K.    
\end{equation*}
\end{proof}

\begin{theorem}\label{t:lower bound}
Assume \eqref{e:alfhors reg}. Then there exist constants $c_L>0$ and $N_1\in\N$ such that for every $N\ge N_1$ and every empirical measure $\mu_N=\frac1N\sum_{i=1}^N\delta_{x_i}$,
\begin{equation}\label{eq:Hdots-lower}
\|\mu_N-\omega\|_{H_0^{-s}} \ \ge\ c_L\,N^{-\frac{1}{2}(1+\frac{q}{\expAhlfors})}.
\end{equation}
\end{theorem}

\begin{proof}
Fix $N\in \N$, to be chosen large enough along the way, and an arbitrary configuration $x_1,\dots,x_N\in\Omega$.
Let $\nu:=\mu_N-\omega\in\calM_0(\R^d)$. Apply Lemma~\ref{l:palle vuote} to obtain $N$ pairwise disjoint balls
$B_{\rho_N}(z_1),\dots,B_{\rho_N}(z_N)$ empty of sample points, thus setting
\begin{equation}\label{e:r-choice}
r_N := \frac{\rho_N}{4}=\frac{\alpha}{4}N^{-\frac1\expAhlfors}.
\end{equation}
the family $\{B_{r_N}(z_j)\}_{j=1}^N$ is empty of sample points as well.

Define
\begin{equation}\label{e:f-def}
f(x):=\sum_{j=1}^N \psi_{r_N,z_j}(x),
\end{equation}
with $\psi_{r,z}$ as in \eqref{e:psi riscalata}.
Since $\mathrm{spt}\,\psi_{r_N,z_j}\subseteq B_{r_N}(z_j)$, we have $f(x_i)=0$ for every $i$, hence
\begin{equation}\label{e:f-muN}
\int_{\R^d} f\,\dd \mu_N = \frac1N\sum_{i=1}^N f(x_i)=0.
\end{equation}

Using \eqref{e:psi-choice} and the definition \eqref{e:psi riscalata},
for $x\in B_{\frac{r_N}{2}}(z_j)$ we have $(x-z_j)/r_N\in B_{\frac12}(0)$ and thus $\psi((x-z_j)/r_{N})= 1$.
Hence
\begin{equation*}
\psi_{r_N,z_j}(x) = r_N^{\,s-\frac d2}\qquad\text{for }x\in B_{\frac{r_N}{2}}(z_j).
\end{equation*}
Therefore, being $\omega$ a positive measure
\begin{equation*}
\int f\,\dd\omega
=\sum_{j=1}^N\int \psi_{r_N,z_j}\,\dd\omega
\ge \sum_{j=1}^N \int_{B_{\frac{r_N}{2}}(z_j)} r_N^{\,s-\frac d2}\,\dd\omega
=r_N^{\,s-\frac d2}\sum_{j=1}^m \omega(B_{\frac{r_N}{2}}(z_j)).     
\end{equation*}
By Ahlfors regularity \eqref{e:alfhors reg} and $z_j\in\mathrm{spt}\,\omega$ (cf. Lemma~\ref{l:palle vuote}), for $N$ large enough so that $r_N\le R_0$, we have
\begin{equation*}
\omega(B_{\frac{r_N}{2}}(z_j))\ge C^{-1}_{\omega} \left(\frac{r_N}2\right)^\expAhlfors.
\end{equation*}
Thus
\begin{equation*}
\int f\,\dd\omega
\ge r_N^{\,s-\frac d2}\,N\,C^{-1}_\omega (r_N/2)^\expAhlfors
= C^{-1}_\omega\left(\frac{ \alpha}{2}\right)^\expAhlfors N^{-\frac1\expAhlfors(\,s-\frac d2)}= C^{-1}_\omega \left(\frac{ \alpha}{2}\right)^\expAhlfors N^{-\frac{q}{2\expAhlfors }}.
\end{equation*}

By Lemma~\ref{l:almost-orth}, since $|z_j-z_i|\geq 4r_N$ holds for every $j\neq i$,
\begin{align*}
& \|f\|_{ H_0^s}=\big\|\sum_{j=1}^N\psi_{r_N,z_j}\big\|_{H_0^s} \leq (\Fhat{C}N)^{\frac12} 
\end{align*}
Define the normalized test function $g:=f\|f\|^{-1}_{H^s_0}\in \mathcal{S}(\R^d,\mathbb{C})$
then 
\begin{equation}\label{e:intg-lower}
\int_{\R^d} g\,\dd\omega = \frac{\int_{\R^d} f\,\dd\omega}{\|f\|_{H_0^s}}
\ge \frac{C^{-1}_\omega\left(\frac{ \alpha}{2}\right)^\expAhlfors N^{-\frac{q}{2\expAhlfors }}}{(\Fhat{C}N)^{\frac12}}.
\end{equation}
Now apply the dual formula in Proposition~\ref{p:dual-norm-exact} for $\nu=\mu_N-\omega$,
\begin{align*}
\|\mu_N-\omega\|_{H_0^{-s}}
&\ge (2\pi)^d\left|\int g\,d(\mu_N-\omega)\right|\\
&=(2\pi)^d\left|\int g\,d\mu_N-\int g\,d\omega\right|
=(2\pi)^d\int g\,d\omega
\ge c_L\,N^{-\frac{1}{2}(1+\frac{q}{\expAhlfors})},    
\end{align*}
with $c_L:=C^{-1}_{\omega}\left(\frac{ \alpha}{2}\right)^\expAhlfors  (2\pi)^d\,\Fhat{C}^{-\frac12}$.
\end{proof}

\subsection{Upper bound}\label{sec: upper bound}
We now turn to the proof of the upper bound. The argument for the main theorem \ref{t:main-intro}
will use, in addition to the Ahlfors regularity assumption \eqref{e:alfhors reg}, the
following connectedness hypothesis
\begin{equation}\tag{Conn}\label{e:supporto connesso}
    \textup{$\textup{spt} \; \omega$ is connected. 
}
\end{equation}
Its role is to invoke the \cite[Theorem 2]{GL17}, which provides equimass cells with a
uniform diameter bound. After completing the connected-support upper bound, we
will explain what remains true for arbitrary Ahlfors regular supports, where one can
still construct equimass partitions but, in general, only with control of the sum of the
$q$-powers of the diameters (cf. Lemma \ref{l:partizione senza connessione}).

In this section $(\Omega,\mathcal{N},\mathbb{P})$ will be a probability space, large enough to guarantees the independence of some random objects. Moreover in addition to assumption \eqref{e:alfhors reg} on $\omega$, we require that

\begin{lemma}\label{l:Gigante-Leopardi}
Assume that $\omega\in \mathcal{P}(\R^d)$ satisfies \eqref{e:alfhors reg} and \eqref{e:supporto connesso}.
Then there exist three constants $c_3,C_4>0$ and $N_0\in \N$, depending only on $d$ and $\omega$, such that for every $N\geq N_0$ there is a family $\{A_1,\dots, A_N\}$ of mutually disjoint Borel sets such that $\cup_{i=1}^NA_i = \textup{spt} \; \omega$, $\omega(A_i)=\frac1N$ and $\diam(A_i)\leq C_4 N^{-\frac1\expAhlfors}$ for every $i=1,\dots,N$ and exist $\{z_1,\dots,z_N\}$ such that $z_i\in A_i$ and $B_{c_3 N^{-\sfrac1\expAhlfors}}(z_i)\cap \mathrm{spt}\, \omega \subseteq A_i$ for every $i=1,\dots,N$.
\end{lemma}

\begin{proof}
Direct consequence of \cite[Theorem 2]{GL17} applied to the metric measure space $(\mathrm{spt}\, \omega, \textup{d}, \omega)$ where $\textup{d}$ is the Euclidean metric.
\end{proof}

Given $N\geq N_2$ and $X_1,\dots,X_N :\Omega \to \R^d$ independent random variables such that $X_i \sim N\omega \res A_i $ for every $i=1,\dots,N$ (cf. Lemma~\ref{l:Gigante-Leopardi}), then we can consider the random empirical measure $\tilde{\mu}_N: \Omega \to \overline{\mathcal{M}_0}$ given by
\begin{equation}\label{e:random empirical}
    \tilde{\mu}_N(\eta):=\frac1N\sum_{i=1}^N (\delta_{X_i(\eta)}-\delta_0)=\frac1N\sum_{i=1}^N\Phi(X_i(\eta)),
\end{equation}
where $\Phi$ is the function in \eqref{e:isometry}.

\begin{lemma}\label{l:polarity}
Let $(\mathcal{K},\langle\cdot,\cdot\rangle_{\mathcal{K}})$ be a separable real Hilbert space and $Z,Z':\Omega \to \mathcal{K}$ be two square-integrable i.i.d $\mathcal{K}$-valued random variables
such that $\mathbb{E}[Z]=0$ and $\mathbb{E}[Z']=0$,  then
\begin{equation*}
\mathbb{E}[\langle Z,Z'\rangle_{\mathcal{K}}]=0.    
\end{equation*}
\end{lemma}

\begin{proof}
Since $Z$ and $Z'$ are independent and both have zero mean,
\begin{equation*}
\mathbb{E}[\langle Z, Z'\rangle_{\mathcal{K}}]
= \langle \mathbb{E}[Z],\, \mathbb{E}[Z']\rangle_{\mathcal{K}}
= \langle 0, 0 \rangle_{\mathcal{K}} = 0,
\end{equation*}
where we used that for independent square-integrable $\mathcal{K}$-valued
random variables the expectation of the inner product factorises as the
inner product of the expectations.
\end{proof}

\begin{lemma}\label{l:hilbert-variance}
Let $(\mathcal{K},\langle\cdot,\cdot\rangle_{\mathcal{K}})$ be a separable real Hilbert space and $Z,Z':\Omega \to \mathcal{K}$ be two square-integrable i.i.d $\mathcal{K}$-valued random variables. Then
\begin{equation}\label{eq:variance-identity}
\mathbb{E}[\|Z-\mathbb{E}[Z]\|_{\mathcal{K}}^2]=\frac12\,\mathbb{E}[\|Z-Z'\|_{\mathcal{K}}^2].
\end{equation}
\end{lemma}
%
\begin{proof}
Since $Z$ and $Z'$ are i.i.d., expanding the square gives
\begin{align*}
\mathbb{E}[\|Z-Z'\|_{\mathcal{K}}^2]
&= \mathbb{E}[\|Z-\mathbb{E}[Z]\|_{\mathcal{K}}^2]
 + \mathbb{E}[\|Z'-\mathbb{E}[Z']\|_{\mathcal{K}}^2]
= 2\,\mathbb{E}[\|Z-\mathbb{E}[Z]\|_{\mathcal{K}}^2],
\end{align*}
where we used that the cross term
$\mathbb{E}[\langle Z - \mathbb{E}[Z],\, Z' - \mathbb{E}[Z']\rangle_{\mathcal{K}}] = 0$
vanishes by independence and zero mean of the centered variables.
Dividing by $2$ gives \eqref{eq:variance-identity}.
\end{proof}

\begin{theorem}\label{t:upper bound in media}
Assume that $\omega\in \mathcal{P}(\R^d)$ satisfies \eqref{e:alfhors reg} and \eqref{e:supporto connesso}. Then there exist $C_L>0$ and $N_0\in \N$ such that for every $N\geq N_0$ 
\begin{equation}\label{e:mean-upper-bound}
\mathbb{E}\bigl[\|\tilde{\mu}_N-\tilde{\omega}\|_{H_0^{-s}}^2\bigr]\ \le\ C_L\,N^{-(1+\frac{q}{\expAhlfors})},
\end{equation}
where $\tilde{\mu}_N$ is given in \eqref{e:random empirical} and $\tilde{\omega}:=\omega-\delta_0\in \calM_0(\R^d)$.
\end{theorem}

\begin{proof}
Fix $N\geq N_0$, with $N_0$ given in Lemma \ref{l:Gigante-Leopardi}, and let $\Phi$ be as in \eqref{e:isometry} and $X_1,\dots,X_N$ as in \eqref{e:random empirical}. We first observe that $\Phi(X_i)\in L^2(\Omega,\overline{\calM_0})$ for every $i=1,\dots,N$, indeed, from Theorem~\ref{t:equivalenza}, for every $\eta\in \Omega$ we have
\begin{equation*}
  \|\Phi(X_i(\eta))\|_{H^{-s}_0}= \|\delta_{X_i(\eta)}-\delta_0\|_{H^{-s}_0} =  \kappa_{d,q}|X_i(\eta)|^q \leq \kappa_{d,q} \textup{dist}(\mathrm{spt}\, \omega,0)^q <\infty, 
\end{equation*}
because $\omega$ has compact support. Setting now, for every $i=1,\dots,N$, $\zeta_i\in L^2(\Omega,\overline{\calM_0})$ as 
\begin{equation*}
    \zeta_i:=\Phi(X_i)-\mathbb E[\Phi(X_i)],
\end{equation*}
we have that $\zeta_1,\dots, \zeta_N : \Omega \to \overline{\mathcal{M}_0}$ are independent $\overline{\calM_0}$-random variables such that $\mathbb{E}[\zeta_i]=0$ for every $i=1,\dots, N$. 

Next, we claim that
\begin{equation}\label{e:mu-omega-expect-rig}
\mathbb E[\|\tilde{\mu}_N-\tilde{\omega}\|_{H_0^{-s}}^2]
=\frac1{N^2}\sum_{i=1}^N \mathbb E[\|\zeta_i\|_{H_0^{-s}}^2]\,.
\end{equation}
To this aim, let $m_i:=N \omega \res A_i \in \calP(\R^d)$ for every $i=1,\dots,N$, 
as $X_i \sim m_i$ we infer 
\begin{equation}\label{e:equa 1 upperbound}
    \int_{\R^d} f \dd m_i = \mathbb{E}[f(X_i)] \qquad \textup{for every $f\in C_b(\R^d,\mathbb{C})$}.
\end{equation}
Note that 
\begin{equation}\label{e:Valore atteso identificato}
\mathbb{E}[\Phi(X_i)] = m_i-\delta_{0}\in \mathcal{M}_0(\R^d).
\end{equation}
Indeed, fix $f\in \mathcal{S}(\R^d,\mathbb{C})$ and by continuity-linearity and \eqref{e:equa 1 upperbound} we have that 
\begin{align}\label{e:uguaglianza valore atteso}
   & \langle f,\mathbb{E}[\Phi(X_i)]\rangle = \mathbb{E}[\langle f,\Phi(X_i)\rangle]= \mathbb{E}[f(X_i)-f(0)]= \\ & \mathbb{E}[f(X_i)]-f(0)= \int_{\R^d}f \;\dd (m_i-\delta_0)= \langle f,m_i-\delta_0\rangle. \nonumber
\end{align}
Therefore, by Corollary \ref{c:Schwartz-separates} we conclude. 

From \eqref{e:Valore atteso identificato}, for every $i=1,\dots,N$, we have that  
\begin{equation*}
\zeta_i
=\Phi(X_i)-\mathbb E[\Phi(X_i)]
=(\delta_{X_i}-\delta_{0})-(m_i-\delta_{0})
=\delta_{X_i}-m_i.    
\end{equation*}
Since $\frac1N\sum_{i=1}^N (m_i-\delta_0)=(\omega-\delta_0)=\tilde{\omega}$, we obtain:
\begin{equation*}
\tilde{\mu}_N-\tilde{\omega}=\frac1N\sum_{i=1}^N(\delta_{X_i}-\delta_0)-\frac1N\sum_{i=1}^N(m_i-\delta_0) = \frac1N\sum_{i=1}^N(\delta_{X_i}-m_i)=\frac1N\sum_{i=1}^N\zeta_i.
\end{equation*}
In particular, we have
\begin{equation*}
  \mathbb E[\|\tilde{\mu}_N-\tilde{\omega}\|_{H^{-s}_0}^2]
=\frac1{N^2}\,\mathbb E\big[\big\|\sum_{i=1}^N\zeta_i\big\|_{H^{-s}_0}^2\big].  
\end{equation*}
Expanding the square and using Lemma~\ref{l:polarity} yields
\begin{equation*}
\mathbb E\big[\big\|\sum_{i=1}^N\zeta_i\big\|_{H^{-s}_0}^2\big]
=\sum_{i=1}^N \mathbb E[\|\zeta_i\|_{H_0^{-s}}^2],   \end{equation*}
hence \eqref{e:mu-omega-expect-rig} follows.

Fix now $i\in \{1,\dots,N\}$. Let $Z=\Phi(X_i)$ and $X_i':\Omega \to \R^d$ be an independent copy of $X_i$, then set $Z':=\Phi(X_i')$.
Lemma~\ref{l:hilbert-variance} and \eqref{e:isometry} give
\begin{align} \label{e:second equa upper}
& \mathbb E[\|\zeta_i\|_{H^{-s}_0}^2]=\mathbb E[\|Z\|_{H^{-s}_0}^2] =\frac12\,\mathbb E[\|Z-Z'\|_{H^{-s}_0}^2]
\\ & =\frac12\,\mathbb E[\|\Phi(X_i)-\Phi(X_i')\|_{H^{-s}_0}^2]=\frac{\kappa_{d,q}}{2}\,\mathbb E[|X_i-X_i'|^q] . \nonumber   
\end{align}
Since $X_i,X_i'\in A_i$ a.s., $|X_i-X_i'|\le \diam(A_i)$ a.s., hence
\begin{equation*}
\mathbb E[\|\zeta_i\|_{H^{-s}_0}^2] \le \frac{\kappa_{d,q}}{2}\,\diam(A_i)^q.    
\end{equation*}
Use \eqref{e:second equa upper} and insert into \eqref{e:mu-omega-expect-rig}:
\begin{equation*}
\mathbb E[\|\tilde{\mu}_N-\tilde{\omega}\|_{H^{-s}_0}^2]
\le \frac1{N^2}\sum_{i=1}^N \frac{\kappa_{d,q}}{2}\,\diam(A_i)^q
\le \frac{\kappa_{d,q}}{2N}\left(\max_{1\le i\le N}\diam(A_i)\right)^q.    
\end{equation*}
Finally, by Lemma~\ref{l:Gigante-Leopardi}, $\diam(A_i)\le C_4 N^{-\frac1\expAhlfors}$, so
\begin{equation*}
\mathbb E[\|\tilde{\mu}_N-\tilde{\omega}\|_{H^{-s}_0}^2] \le \frac{\kappa_{d,q}}{2N}C_4^q\,N^{-\frac{q}{\expAhlfors}}
= C_L\,N^{-(1+\frac{q}{\expAhlfors})},    
\end{equation*}
where $C_L:= \frac{\kappa_{d,q}}{2}C_4^q$.
\end{proof}

\begin{corollary}
Assume that $\omega\in \mathcal{P}(\R^d)$ satisfies \eqref{e:alfhors reg} and \eqref{e:supporto connesso}, then there exists $C_L>0$ and $N_0\in \N$ such that for every $N\geq N_0$ there exist $x_1,\dots,x_N \in \textup{spt} \; \omega$ with
\begin{equation}\label{e:upper-bound-cor}
\|\mu_N-\omega\|_{H_0^{-s}}\le\ C_L\,N^{-\frac{1}{2}(1+\frac{q}{\expAhlfors})},
\end{equation}
where ${\mu_N}:=\frac1N\sum_{i=1}^N \delta_{x_i}$. 
\end{corollary}

We now pass to consider the general Ahlfors regular case,
where no connectedness assumption is imposed on $\mathrm{spt}\,\omega$. The
price to pay is that the Gigante--Leopardi uniform estimate
$\textup{diam}(A_i)\lesssim N^{-1/\expAhlfors}$ is no longer available in
general. However, the probabilistic proof only requires control of
$\sum_i \operatorname{diam}(A_i)^q$. The next lemma provides precisely this weaker
global estimate.
\begin{lemma}\label{l:partizione senza connessione}
Assume that $\omega\in \mathcal{P}(\R^d)$ satisfies \eqref{e:alfhors reg}. Then, there exist a constant $\hat C > 0$ and $N_0 \in \mathbb{N}$ such that for every integer $N \ge N_0$, one can find a Borel partition $\textup{spt} \; \omega = A_1 \sqcup \cdots \sqcup A_N$ satisfying
\begin{equation*}
    \omega(A_i) = \frac1N \quad \text{for all } i=1,\dots,N,  
\end{equation*}
and
\begin{equation*}
     \sum_{i=1}^N \operatorname{diam}(A_i)^q \le \hat C
    \begin{cases} 
    N^{1-\frac{q}{\expAhlfors}}, & \text{if } q < \expAhlfors,\\[2mm] 
    \log N, & \text{if } q = \expAhlfors,\\[2mm] 
    1, & \text{if } q > \expAhlfors. 
    \end{cases}
\end{equation*}
\end{lemma}

\begin{proof}
Let $K:=\mathrm{spt}\, \omega$. From Remark \ref{r:supporto compatto} we have that $K$ is compact and thus we can embed it into a half-open cube $Q^0 \subset \mathbb{R}^d$ of side length $\ell_0 > 0$. For each integer $k \ge 0$, let $\mathcal{D}_k$ denote the grid of half-open dyadic subcubes of $Q^0$ with side length $\ell_k := 2^{-k}\ell_0$. This induces a natural partition of $K$ at each generation $k$, given by $K = \bigsqcup_{Q\in\mathcal{D}_k} K_Q$, where $K_Q := K \cap Q$. 

Let $\mathcal{A}_k := \{Q \in \mathcal{D}_k : K_Q \neq \varnothing\}$ denote the family of active cubes at generation $k$. We claim that for $k$ sufficiently large, 
\begin{equation}\label{e:cardinal Ak}
\#\mathcal{A}_k \le C_1 \ell_k^{-\expAhlfors},    
\end{equation}
for some $C_1>0$. To see this, extract a maximal $\ell_k$-separated set $Z_k \subset K$. The maximality ensures $K \subset \bigcup_{z\in Z_k} B_{\ell_k}(z)$, while the separation implies that the balls $\{B_{\ell_k/2}(z)\}_{z\in Z_k}$ are pairwise disjoint. For $k$ large enough such that $\ell_k/2 < R_0$, the lower Ahlfors bound yields $\omega(B_{\ell_k/2}(z)) \ge C^{-1}_{\omega}(\ell_k/2)^\expAhlfors $. Since $\omega$ is a probability measure, it follows that $\# Z_k \le \frac{2^\expAhlfors }{c} \ell_k^{-\expAhlfors}$. The maximality ensures $K \subset \bigcup_{z\in Z_k} B_{\ell_k}(z)$, therefore, since each ball $B_{\ell_k}(z)$ intersects at most $C_d$ dyadic cubes of side length $\ell_k$, the bound on $\#\mathcal{A}_k$ immediately follows.

Given $N\in \N$, we define  $J = J(N)\in \N$ as the smallest positive integer such that $ C(\sqrt{d}\,\ell_J)^\expAhlfors = C\ell_0^\expAhlfors (\sqrt{d}\,2^{-J})^\expAhlfors  \le \frac1N$, hence if $N$ is large enough we get $\sqrt{d}\,\ell_J < R_0$ and $J(N)>1$. In particular, by the minimality of $J$, we have that 
\begin{equation}\label{e:Def di J}
    C\ell_0^\expAhlfors (\sqrt{d}\,2^{-J+1})^\expAhlfors  > \frac1N
\end{equation}
therefore there exists a constant $\overline{C}$, independent of $N$, such that 
\begin{equation}\label{e:logN J}
    J=J(N)\in( \overline{C}^{-1} \log N,\overline{C} \log N).
\end{equation}
Consequently, for any cube $Q \in \mathcal{A}_J$, picking any $x_Q \in K_Q$ gives $K_Q \subset B_{\sqrt{d}\,\ell_J}(x_Q)$, which by the upper Ahlfors bound guarantees $\omega(K_Q) \le \frac1N$.

Fixed $N\in \N$, we construct the partition inductively, moving from the terminal generation $J(N)$ upward to the root $Q^0$. For each active cube $Q \in \mathcal{A}_J$, we initialize its residue as $R_Q := K_Q$, and from the previous argument we have $\omega(R_Q)\leq \frac1N$. We proceed by backward induction on $k < J$. At generation $k$, for each $Q \in \mathcal{A}_k$, assume that for all active sub-cubes $Q' \in \operatorname{Sub}(Q) \subset \mathcal{A}_{k+1}$ we have already defined valid residues $R_{Q'} \subset K_{Q'}$ such that $\omega(R_{Q'}) \le \frac1N$. Then we define the pooled residue as
\begin{equation*}
    S_Q := \bigcup_{Q'\in\operatorname{Sub}(Q)} R_{Q'}.
\end{equation*}
Since a dyadic cube in $\mathbb{R}^d$ has exactly $2^d$ direct sub-cubes, $\omega(S_Q) \le 2^d / N$.

Let $m_Q := \lfloor N\omega(S_Q) \rfloor$. Now, by the Ahlfors regularity \eqref{e:alfhors reg} we have that $\omega$ is non-atomic, and therefore we can extract exactly $m_Q$ pairwise disjoint Borel sets $A_{Q,1}, \dots, A_{Q,m_Q} \subset S_Q$, each of mass exactly $\frac1N$ (see \cite[Corollary 1.21]{fonseca2007}). The updated residue for $Q$ is then defined as the remainder:
\begin{equation*}
    R_Q := S_Q \setminus \bigcup_{j=1}^{m_Q} A_{Q,j}.
\end{equation*}
By the definition of $m_Q$ we have that $\omega(R_Q) = \omega(S_Q) - \frac{1}{N}m_Q < \frac1N$. Furthermore, since each set $A_{Q,j}$ is contained in $Q$, we have $\operatorname{diam}(A_{Q,j}) \le \operatorname{diam}(Q) = \sqrt{d}\,\ell_k$.

This process propagates up to the root $Q^0$ at $k=0$. Let $\mathcal{F}$ denote the global collection of all the pairwise disjoint sets $A_{Q,j}$ extracted during the entire process across all scales. Let $M$ be the total number of sets in $\mathcal{F}$. The construction terminates with a final root residue $R_{Q^0}$, which satisfies $\omega(R_{Q^0}) < \frac1N$.

By mass conservation, since the original support $K$ is partitioned exactly into the $M$ sets of $\mathcal{F}$ and the final residue $R_{Q^0}$, we can write:
\begin{equation*}
    1 = \omega(K) = \sum_{A \in \mathcal{F}} \omega(A) + \omega(R_{Q^0}) = \frac MN + \omega(R_{Q^0}).
\end{equation*}
Rearranging this identity yields:
\begin{equation*}
    \omega(R_{Q^0}) = \frac{N - M}{N}.
\end{equation*}
Because the inductive process ensures $\omega(R_{Q^0}) < \frac1N$, we deduce that $(N - M)/N < \frac1N$, in turn implying so that $N - M < 1$, and finally $M = N$ since both $N$ and $M$ are non-negative integers.
Consequently, $\omega(R_{Q^0}) = 0$, proving that we have extracted exactly $N$ sets. Adding the null residue $R_{Q^0}$ to the first set $A_1$ produces the desired exact Borel partition $K = A_1 \sqcup \cdots \sqcup A_N$, with $\omega(A_i) = \frac1N$ for all $i$.

During the construction, at each step $k$, each cube $Q \in \mathcal{A}_k$ generates at most $2^d$ sets, each bounded in diameter by $\sqrt{d}\,\ell_k$. Summing over all scales, thanks to \eqref{e:cardinal Ak}, there exists a constant $\tilde C$, such that
\begin{align*}
& \sum_{i=1}^N \textup{diam}(A_i)^q\leq \textup{diam}(K)^q+ \sum_{i=2}^N \textup{diam}(A_i)^q \\ & \le \tilde C+ \tilde C \sum_{k=1}^{J(N)-1} (\#\mathcal{A}_k)\,\ell_k^q \le \tilde C \sum_{k=0}^{J(N)-1} \ell_k^{q-\expAhlfors}=  \tilde C\ell_0^{q-\expAhlfors} \sum_{k=0}^{J(N)-1} 2^{k(\expAhlfors-q)}.    
\end{align*}
In particular, if $q < \expAhlfors$, from \eqref{e:logN J} we get
\begin{equation*}
  \sum_{i=1}^N \textup{diam}(A_i)^q\leq \frac{\tilde C\ell_0^{q-\expAhlfors}}{2^{\expAhlfors-q}-1}2^{J(N)(\expAhlfors-q)}\leq  \hat C N^{\frac{\expAhlfors-q}{\expAhlfors}}=\hat C N^{1-\frac{q}{\expAhlfors}}.   
\end{equation*}
Instead, if $q = \expAhlfors$, again by \eqref{e:logN J} we obtain
\begin{equation*}
  \sum_{i=1}^N \textup{diam}(A_i)^q\leq \tilde C J(N)\leq  \hat C \log(N).
\end{equation*}
Eventually, in case $q > \expAhlfors$, the series $\sum_{k=0}^{\infty} 2^{k(\expAhlfors-q)}$ is convergent. This completes the proof.
\end{proof}

\begin{corollary}\label{c:upper senza connesso}
Assume that $\omega\in \mathcal{P}(\R^d)$ satisfies \eqref{e:alfhors reg}, then for every $N\geq N_0$ there exists $x_1,\dots,x_N \in \textup{spt} \; \omega$ such that 
\begin{equation}\label{e:upper-bound-non connesso}
\|\mu_N-\omega\|_{H_0^{-s}}\le\ C_L 
\begin{cases} 
    N^{-\frac12(1+\frac{q}{\expAhlfors})}, & \text{if } q < \expAhlfors,\\[2mm] 
    N^{-1}(\log N)^{\frac12}, & \text{if } q = \expAhlfors,\\[2mm] 
    N^{-1}, & \text{if } q > \expAhlfors. 
    \end{cases}
\end{equation}
where ${\mu_N}:=\frac1N\sum_{i=1}^N \delta_{x_i}$. 
\end{corollary}

\begin{proof}
The proof is entirely analogous to the proof of Theorem~\ref{t:upper bound in media}. Indeed,
one repeats the same probabilistic construction by sampling independently one point
in each element of an equimass partition of $\mathrm{spt}\,\omega$. The only
difference is that, instead of using the uniform diameter estimate provided by \cite[Theorem 2]{GL17} (cf. Lemma~\ref{l:Gigante-Leopardi}), we use the partition constructed in
Lemma~\ref{l:partizione senza connessione}. Thus the argument of Theorem~\ref{t:upper bound in media}
gives
\begin{equation*}
    \mathbb E\Big[ \|\mu_N-\omega\|_{H_0^{-s}}^2\Big]\le\frac{\kappa_{d,q}}{2N^2}
    \sum_{i=1}^N \textup{diam}(A_i)^q .
\end{equation*}
Applying Lemma~\ref{l:partizione senza connessione}, we obtain
\begin{equation*}
    \mathbb E\Big[
        \|\mu_N-\omega\|_{H_0^{-s}}^2
    \Big]
    \le C
    \begin{cases}
        N^{-(1+\frac q{\expAhlfors})}, & \text{if } q<\expAhlfors,\\[1mm]
        N^{-2}\log N, & \text{if } q=\expAhlfors,\\[1mm]
        N^{-2}, & \text{if } q>\expAhlfors.
    \end{cases}
\end{equation*}
Therefore there exists at least one realization of the random variables for which the
same bound holds, giving \eqref{e:upper-bound-non connesso} after taking the square root.  
\end{proof}

The previous corollary shows that, without the connectedness assumption, the upper bound may deteriorate in the critical and supercritical regimes.
We prove next that this loss is not merely a technical point of the proof. The obstruction is present for any probability measure with compact and disconnected support. In particular, if $\expAhlfors\in(0,2)$ and $\omega$ satisfies \eqref{e:alfhors reg}, the result below shows that for every $q\in(\expAhlfors,2)$ the upper bound obtained in Corollary~\ref{c:upper senza connesso} is optimal in case the support of $\omega$ is not connected. In addition, the same holds true for any probability measure $\omega$ with disconnected support in view of 
\eqref{eq:stat}.
\begin{proposition}
    Assume that $\omega\in \mathcal{P}(\R^d)$ with $\mathrm{spt}\,\omega$ compact and disconnected. Then there exist a subsequence of integers $(N_j)_{j\in\mathbb N}$ and a constant $C>0$ such that 
    \[
     \|\mu_{N_j}-\omega\|_{H_0^{-s}}\geq \frac{C}{N_j}
    \]
    for every empirical measure $\mu_{N_j}=\frac1{N_j}\sum_{i=1}^{N_j}\delta_{x_i}$.
\end{proposition}
\begin{proof}
Let $U$ and $V$ be disjoint and forming a non-trivial open cover (in the relative topology) of $\mathrm{spt}\,\omega$. In turn, $U$ and $V$ are closed and thus compact. Moreover, both have positive measure with respect to $\omega$.

Let $\delta:=\frac14\mathrm{dist}(U,V)>0$, 
$U_\delta:=\{x\in \R^d:\,\mathrm{dist}(x,U)<\frac\delta4\}$ and $V_\delta$ defined analogously with $V$ in place of $U$. Then, $U_\delta$ and 
$V_\delta$ are disjoint open sets such that $U\subset\subset U_\delta$
and $V\subset\subset V_\delta$, with $U_\delta\cap V_\delta=\emptyset$.

Consider next functions $\varphi_1\in C^\infty_c(U_{2\delta},[0,1])$ with 
$\varphi_1\equiv 1$ on $U_\delta$, and  $\varphi_2\in C^\infty_c(B_{\sfrac\delta8}[0,1])$, with $\varphi_2\equiv 1$ on $B_{\sfrac\delta{16}}$. 
Let $\mu_N=\frac1N\sum_{i=1}^N\delta_{x_i}$ be any empirical measure.

We first study the case in which there is (at least) one point 
$x_i\in \R^d\setminus (U_\delta\cup V_\delta)$, 
$i\in\{1,\ldots, N\}$. 
By applying the duality formula in Proposition~\ref{p:dual-norm-exact},
we then estimate as follows
\begin{align*}
\|\mu_N-\omega\|_{H_0^{-s}}
\ge (2\pi)^d\left|\int \frac{\varphi_2(\cdot-x_i)}{\|\varphi_2\|_{H_0^s}}\,d(\mu_N-\omega)\right|
\ge\frac{(2\pi)^d}{\|\varphi_2\|_{H_0^s}N}\,.
\end{align*}

We conclude by considering the remaining case in which $\{x_1,\ldots, x_N\}\subset U_\delta\cup V_\delta$.
By applying the duality formula in Proposition~\ref{p:dual-norm-exact},
we get
\begin{align*}
\|\mu_N-\omega\|_{H_0^{-s}}
\ge (2\pi)^d\left|\int \frac{\varphi_1}{\|\varphi_1\|_{H_0^s}}\,d(\mu_N-\omega)\right|
=\frac{(2\pi)^d}{\|\varphi_1\|_{H_0^s}}\left|\frac{\kappa_N}{N}-\omega(U)\right|\,,    
\end{align*}
where $\kappa_N:= \#\{i\in \{1,\dots, N\} \; : \; x_i\in U_\delta\}
\in\{0,\ldots,N\}$.

We claim that there exist a subsequence of integers $(N_j)_{j\in\mathbb N}$ and a constant $C>0$ such that for every $j\in\mathbb N$
\[
\left|\frac{\kappa_{N_j}}{N_j}-\omega(U)\right|\ge \frac{C}{N_j}.
\]
Indeed, if $\omega(U)=\frac{m}{n}\in\mathbb{Q}$, where $m$ and $n$ are relatively prime integers, it suffices to choose $N_j:=jn+1$ to infer 
\[
\left|\frac{\kappa_{N_j}}{N_j}-\omega(U)\right|\ge \frac{1}{N_j}\min\{\omega(U),\omega(V)\}.
\]
Otherwise, if $\omega(U)\in\mathbb{R}\setminus\mathbb{Q}$ assume by contradiction that $|\kappa_N-N\omega(U)|\to0$ as $N\to\infty$.
In particular, $\mathrm{dist}(N\omega(U),\mathbb{Z})\to0$ as $N\to\infty$.
Thus, estimating
\[
\mathrm{dist}(\omega(U),\mathbb{Z})\leq
\mathrm{dist}((N+1)\omega(U),\mathbb{Z})+
\mathrm{dist}(-N\omega(U),\mathbb{Z})\,,
\]
by letting $N\to\infty$ we conclude that $\omega(U)\in\mathbb{Z}$, a contradiction.
\end{proof}

\section{Sharp comparison between \texorpdfstring{$\mathcal{E}_q$}{Eq} 
and the Wasserstein distance}\label{sec:W1}
The energy distance $\mathcal{E}_q$ and the Wasserstein distances $W_p$ 
induce comparable (weak) topologies on $\mathcal{P}(\R^d)$, as established in 
Proposition~\ref{prop:edtop}.  Related quantitative comparisons between 
Wasserstein distances and Fourier-based probability metrics have been studied in 
the kinetic and imaging literature; see, for instance, 
\cite{Auricchio2020}, where equivalence estimates 
with $W_1$ and $W_2$ are proved in a discrete setting, and the broader discussion 
in \cite{auricchio2026kinetic}.

In this section we sharpen this qualitative 
comparison to a quantitative inequality, showing that  $\mathcal{E}_q^2(\mu,\nu) \leq  W_1^q(\mu,\nu)$, with equality  characterized exactly. This estimate is sharp -- equality holds precisely  for pairs of Dirac masses -- and implies in particular that  $\mathcal{E}_q^2 \leq  W_1^q \leq  W_q^q$. To the best of our knowledge, this estimate has previously appeared only in the case $q = 1$ \cite[Theorem~2]{hertrich2024generative}, \cite[Lemma~5]{hagemann2024posterior}. The extension to general $q \in (0,2)$ and the characterization of equality cases appear to be new.
\begin{theorem}\label{thm:W1}
 Let $\mu,\nu \in \mathcal{P}(\R^d)$ with compact support then 
 \begin{equation}\label{e:W_1 comparison}
     \|\mu-\nu\|^2_{H_0^{-s}}\leq \kappa_{d,q}\; W^q_1(\mu,\nu),
 \end{equation}
 where $W_1(\cdot,\cdot)$ is the $1$-Wasserstein distance and $\kappa_{d,q}$ is the constant given in Theorem~\ref{t:main-intro}. 
 Furthermore the equality in \eqref{e:W_1 comparison} holds if and only if either $\mu=\nu$ or $\mu=\delta_{x_0}$ and $\nu=\delta_{y_0}$ for some $x_0,y_0\in \R^d$.
\end{theorem}

\begin{proof}
    Let $\pi\in\Pi(\mu,\nu)$ be an optimal coupling for $W_1$, and let $X,Y:\Omega\to \R^d$ be two random variables such that $(X,Y)\sim \pi$, and thus
\begin{equation*}
W_1(\mu,\nu)=\mathbb{E}[|X-Y|].    
\end{equation*}
Using the same argument in \eqref{e:uguaglianza valore atteso} we have that $\mathbb{E}[\Phi(X)]=\mu-\delta_0$ and $\mathbb{E}[\Phi(Y)]=\nu-\delta_0$, therefore
\begin{align}\label{e:eq1 wasserstein}
&\|\mu-\nu\|_{H^{-s}_0}=\big\|\E[\Phi(X)]-\E[\Phi(Y)]\big\|_{H^{-s}_0}
\\ & =\big\|\E[\Phi(X)-\Phi(Y)]\big\|_{H^{-s}_0}  \leq \E[\|\Phi(X)-\Phi(Y)\|_{H^{-s}_0}]. \nonumber
\end{align}
Furthermore, from \eqref{e:isometry}, we have 
\begin{equation*}
\E[\|\Phi(X)-\Phi(Y)\|_{H^{-s}_0}]=\kappa_{d,q}^{\frac12}\E[|X-Y|^{q/2}].    
\end{equation*}
Since $q/2\in(0,1)$, the function $t\mapsto t^{q/2}$ is concave on $[0,\infty)$, hence Jensen's inequality gives
\begin{equation}\label{e:eq 2 wasserstein}
\E[|X-Y|^{q/2}]\ \le\ (\E[|X-Y|])^{q/2}=W^{q/2}_1(\mu,\nu).    
\end{equation}
Combining the last three equations proves \eqref{e:W_1 comparison}.

We now characterize the equality cases. If $\mu=\nu$, then
\begin{equation*}
\|\mu-\nu\|_{H_0^{-s}}=0
\qquad\text{and}\qquad
W_1(\mu,\nu)=0,    
\end{equation*}
so equality holds. If instead $\mu=\delta_x$ and $\nu=\delta_y$, then
\begin{equation*}
  W_1(\delta_x,\delta_y)=|x-y|,  
\end{equation*}
and therefore by \eqref{e:isometry} we have
\begin{equation*}
\|\delta_x-\delta_y\|_{H_0^{-s}}^2
= \kappa_{d,q}|x-y|^q = \kappa_{d,q}W_1(\delta_x,\delta_y)^q.    
\end{equation*}
Thus equality also holds in this case. It remains to prove the converse. Suppose that
\begin{equation}\label{e:caso uguaglianza}
\|\mu-\nu\|_{H_0^{-s}}^2=\kappa_{d,q}W_1(\mu,\nu)^q.    
\end{equation}
Set $m:=W_1(\mu,\nu)=\mathbb E [|X-Y|]$. If $m=0$ then, since $W_1$ is a distance, we have $\mu=\nu$. Assume now that $m>0$. From \eqref{e:eq1 wasserstein} and \eqref{e:eq 2 wasserstein}, all
intermediate inequalities must be equalities. 
Hence
\begin{equation}\label{e:uguaglianza dis norma}
\|\mathbb E [\Phi(X)-\Phi(Y)]\|_{H_0^{-s}}=\mathbb E[\|\Phi(X)-\Phi(Y)\|_{H_0^{-s}}],    
\end{equation}
and
\begin{equation}\label{e:uguaglianza jensen}
\mathbb E[|X-Y|^{\frac{q}{2}}]=(\mathbb E[|X-Y|])^{\frac{q}{2}}.    
\end{equation}
Thanks to the strictly concavity of \(t\mapsto t^{\frac{q}{2}}\), equality \eqref{e:uguaglianza jensen} in Jensen's inequality is
possible only if $|X-Y|$ is almost surely constant. Consequently
\begin{equation}\label{e:cost differenza}
|X(\omega)-Y(\omega)|=m \quad    
\end{equation}
for $\mathbb{P}$-a.e. $\omega\in \Omega$. From \eqref{e:eq1 wasserstein}, \eqref{e:caso uguaglianza}, \eqref{e:uguaglianza dis norma}, \eqref{e:cost differenza} and \eqref{e:isometry} we have 
\begin{equation*}
    \|\E[\Phi(X)-\Phi(Y)]\|_{H^{-s}_0}=\kappa_{d,q}^{1/2}m^{\frac{q}{2}}>0.
\end{equation*}
Then setting 
\begin{equation*}
w:=\frac{\mathbb E [\Phi(X)-\Phi(Y)]}{\|\mathbb E [\Phi(X)-\Phi(Y)]\|_{H_0^{-s}}}, 
\end{equation*}
we have
\begin{align*}
& 0=\mathbb E[\|\Phi(X)-\Phi(Y)\|_{H_0^{-s}}]-
\|\mathbb E[\Phi(X)-\Phi(Y)]\|_{H_0^{-s}} \\ & =\mathbb E[\|\Phi(X)-\Phi(Y)\|_{H_0^{-s}}]
-\left\langle \mathbb E[\Phi(X)-\Phi(Y)],w\right\rangle_{H_0^{-s}} \\ &=\mathbb E[
\|\Phi(X)-\Phi(Y)\|_{H_0^{-s}}
-\left\langle \Phi(X)-\Phi(Y),w\right\rangle_{H_0^{-s}}
],
\end{align*}
where 
\begin{equation*}
\|\Phi(X)-\Phi(Y)\|_{H_0^{-s}}
-
\left\langle \Phi(X)-\Phi(Y),w\right\rangle_{H_0^{-s}}\geq 0,    
\end{equation*}
by Cauchy-Schwartz inequality. Since the expectation is zero, it must vanish almost surely.
Therefore
\begin{equation*}
\left\langle \Phi(X)-\Phi(Y),w\right\rangle_{H_0^{-s}}
=
\|\Phi(X)-\Phi(Y)\|_{H_0^{-s}}
\end{equation*}
and consequently
\begin{equation*}
\Phi(X)-\Phi(Y)=\|\Phi(X)-\Phi(Y)\|_{H_0^{-s}}w
\end{equation*}
almost surely. 
Since $\|\Phi(X)-\Phi(Y)\|_{H_0^{-s}}$ is almost surely constant, from \eqref{e:cost differenza} and \eqref{e:isometry}, thus $\Phi(X)-\Phi(Y)$ itself is
almost surely constant. In particular there exists an element $h\in \overline{\mathcal{M}_0}$ such that
\begin{equation*}
\Phi(X)-\Phi(Y)=\delta_X-\delta_Y=h
\qquad\text{almost surely.}    
\end{equation*}
Define the closed set
\begin{equation*}
   S:=
\left\{
(x,y)\in \R^d\times\R^d:
\delta_x-\delta_y=h,
\quad
|x-y|=m
\right\}. 
\end{equation*}
From the previous argument,
\begin{equation*}
\pi(S)=1.    
\end{equation*}
In particular $S$ is non-empty, but from it's definition we have that $S=\{(x_0,y_0)\}$ for some $x_0,y_0\in \R^d$ because for every $x,x',y,y' \in \R^d$ we have that $\delta_x-\delta_y=\delta_{x'}-\delta_{y'}$ implies $x=x'$ and $y=y'$. Therefore $\pi=\delta_{(x_0,y_0)}$ and $\mu=\delta_{x_0}$, $\nu=\delta_{y_0}$, concluding the proof.
\end{proof}

\begin{center}
  \FundingLogos
  \end{center}
  \vspace{0.5em}
  \begin{tcolorbox}\centering\small
   
    Funded by the European Union. Views and opinions expressed are however those of the author(s) only and do not necessarily reflect those of the European Union or the European Research Council Executive Agency. Neither the European Union nor the granting authority can be held responsible for them. This project has received funding from the European Research Council (ERC) under the European Union’s Horizon Europe research and innovation programme (grant agreement No. 101198055, project acronym NEITALG).
    
  \end{tcolorbox}

\bibliographystyle{plain}
\bibliography{bib}

\end{document}